\pdfoutput=1
\RequirePackage{ifpdf}
\ifpdf 
\documentclass[pdftex]{sigma}
\else
\documentclass{sigma}
\fi

\numberwithin{equation}{section}

\newtheorem{Theorem}{Theorem}[section]

\newtheorem{Lemma}[Theorem]{Lemma}
\newtheorem{Proposition}[Theorem]{Proposition}
 { \theoremstyle{definition}

\newtheorem{Remark}[Theorem]{Remark} }

\newcommand{\N}{\mathbb{N}}
\newcommand{\C}{\mathbb{C}}
\newcommand{\R}{\mathbb{R}}
\newcommand{\Z}{\mathbb{Z}}
\newcommand{\RR}{\mathbb{R}^d}
\newcommand{\RRR}{\mathbb{R}^d_+}

\def\eps{\varepsilon}

\def\om{\omega}
\def\a{\alpha}

\def\t{\theta}

\def\z{\zeta}
\def\dersym{\delta^{\textrm{sym}}}
\def\dersymstate{\delta^{\textrm{\emph{sym}}}}

\DeclareMathOperator{\id}{Id}

\DeclareMathOperator{\e}{{\mathbb{E}\textrm{{xp}}}}
\DeclareMathOperator{\ee}{{\mathbb{E}\textrm{\emph{xp}}}}
\DeclareMathOperator{\lo}{{\mathbb{L}\textrm{og}}}
\DeclareMathOperator{\loo}{{\mathbb{L}\textrm{\emph{og}}}}
\DeclareMathOperator{\Id}{Id}

\newcommand{\B}{\mathbb{B}}
\DeclareMathOperator{\supp}{supp}

\def\piint{\Pi_{\alpha + \eta + \boldsymbol{1} + \varepsilon}(ds)}
\def\piinta{\Pi_{\alpha + \eta + 1 + \varepsilon}(ds)}
\def\bb{(\z,q_{\pm})}

\def\bbb{(\z,q_{\pm}(x,y,s))}

\def\eee{(\z,q_{\pm}({x \vee x'},y,s))}
\def\fff{(\z,q_{\pm}(x,y \vee y',s))}
\def\zz{(\z,q_{\pm}(x+z,y,s))}

\def\1{\Pi_{\alpha+\mathbf{1}+\eps}(ds)}

\begin{document}


\newcommand{\arXivNumber}{1602.05369}

\renewcommand{\PaperNumber}{096}

\FirstPageHeading

\ShortArticleName{On Harmonic Analysis Operators in Laguerre--Dunkl and Laguerre-Symmetrized Settings}

\ArticleName{On Harmonic Analysis Operators in Laguerre--Dunkl\\ and Laguerre-Symmetrized Settings}

\Author{Adam NOWAK~$^\dag$, Krzysztof STEMPAK~$^\ddag$ and Tomasz Z.~SZAREK~$^\dag$}

\AuthorNameForHeading{A.~Nowak, K.~Stempak and T.Z.~Szarek}

\Address{$^\dag$~Institute of Mathematics, Polish Academy of Sciences, \\
\hphantom{$^\dag$}~\'Sniadeckich 8, 00--656 Warszawa, Poland}
\EmailD{\href{mailto:adam.nowak@impan.pl}{adam.nowak@impan.pl}, \href{mailto:szarektomaszz@gmail.com}{szarektomaszz@gmail.com}}

\Address{$^\ddag$~Faculty of Pure and Applied Mathematics, Wroc\l{}aw University of Science and Technology, \\
\hphantom{$^\ddag$}~Wyb{.}\ Wyspia\'nskiego 27, 50--370 Wroc\l{}aw, Poland}
\EmailD{\href{mailto:krzysztof.stempak@pwr.edu.pl}{krzysztof.stempak@pwr.edu.pl}}

\ArticleDates{Received May 25, 2016, in f\/inal form September 23, 2016; Published online September 29, 2016}

\Abstract{We study several fundamental harmonic analysis operators in the multi-dimen\-sional context of the Dunkl
harmonic oscillator and the underlying group of ref\/lections isomorphic to $\mathbb{Z}_2^d$. Noteworthy,
we admit negative values of the multiplicity functions. Our investigations include maximal operators,
$g$-functions, Lusin area integrals, Riesz transforms and multipliers of Laplace and Laplace--Stieltjes type.
By means of the general Calde\-r\'on--Zygmund theory we prove that these operators are bounded on weighted
$L^p$ spaces, \mbox{$1 < p < \infty$}, and from weighted $L^1$ to weighted weak $L^1$. We also obtain similar
results for analogous set of operators in the closely related multi-dimensional Laguerre-symmetrized
framework. The latter emerges from a symmetrization procedure proposed recently by the f\/irst two authors.
As a by-product of the main developments we get some new results in the multi-dimensional Laguerre function
setting of convolution type.}

\Keywords{Dunkl harmonic oscillator; generalized Hermite functions; negative multiplicity function; Laguerre expansions of convolution type; Bessel harmonic oscillator; Laguerre--Dunkl expansions; Laguerre-symmetrized expansions; heat semigroup; Poisson semigroup; maximal operator; Riesz transform; $g$-function; spectral multiplier; area integral; Calder\'on--Zygmund operator}

\Classification{42C99; 42C10; 42C20; 42B20;	42B15; 42B25}

\section{Introduction} \label{sec:intro}

Analysis related to Dunkl operators, that is dif\/ferential-dif\/ference operators associated with ref\/lection groups, is an important development in modern mathematics. For basic concepts of this theory see Dunkl's pioneering work \cite{D} and, for instance, the survey article by R\"osler~\cite{R}. Harmonic analysis related to the Dunkl harmonic oscillator (DHO in short) has been intensively studied in recent years by the authors \cite{NoSt4,NoSt3,NoSt6,TZS2,TZS3} and many other mathematicians, see, e.g., \cite{A,AT,BSa,BRT,BT,N,W} and references therein. A commonly appearing assumption in this literature is that the underlying multiplicity function is non-negative. The latter postulate is critical in several important aspects of the Dunkl theory, like existence of a convolution structure or existence and uniqueness of the intertwining operator. Nevertheless, there are few papers like~\cite{AC1,AC2} and the very recent paper~\cite{NoSt7} by the f\/irst two authors, where certain harmonic analysis aspects for the one-dimensional DHO are successfully studied with negative values of the multiplicity function admitted. This in a large extent motivated our present research since then the following question naturally arises. Does it make sense to study basic harmonic analysis operators in the context of the multi-dimensional DHO when the underlying multiplicity function takes also negative values? The answer we f\/ind is def\/initely af\/f\/irmative, at least in the special case when the associated group of ref\/lections is isomorphic to~$\mathbb{Z}_2^d$. It occurs that negative multiplicity functions have the same rights in the sense of results obtained, but not in the sense of the associated analysis which is in this case more sophisticated and requires more subtle techniques. Similar conclusions were drawn recently in the more elementary Dunkl Laplacian context and the same group of ref\/lections by Castro and the third-named author~\cite{CaSz}. The latter paper was in fact an important inspiration for the present one.

The context of the DHO with the underlying group of ref\/lections isomorphic to $\mathbb{Z}_2^d$ (written shortly $[{\rm DHO},\mathbb{Z}_2^d]$) is sometimes called the Laguerre--Dunkl setting. The reason is that this Dunkl framework is deeply connected with the situation of Laguerre function expansions of convolution type, in particular the standard eigenfunctions of the $[{\rm DHO},\mathbb{Z}_2^d]$ are directly related to Laguerre functions of convolution type.
The results we obtain in the Laguerre--Dunkl context, see Theorem~\ref{thm:main}, are weighted $L^p$ mapping properties of several fundamental harmonic analysis operators like the heat semigroup maximal operator, mixed $g$-functions of arbitrary orders, mixed Lusin area integrals of arbitrary orders, Riesz transforms of arbitrary order, and spectral multipliers of Laplace and Laplace--Stieltjes transform type (these multiplier operators cover, as special cases, imaginary powers of the DHO and fractional integrals related to the DHO, respectively). This brings a number of new results that extend those existing in the literature by considering non-positive multiplicity functions and more general operators. This concerns, in particular, the authors' papers~\cite{NoSt4, NoSt3} where heat semigroup maximal operator, f\/irst order Riesz transforms and imaginary powers of the $[{\rm DHO},\mathbb{Z}_2^d]$ were studied and~\cite{TZS2,TZS3}, where the Laplace type multipliers of both types and f\/irst order $g$-functions and Lusin area integrals for the $[{\rm DHO},\mathbb{Z}_2^d]$ were investigated. Further, we generalize in the above mentioned directions and, in addition, by allowing weights, the results of Ben Salem and Samaali~\cite{BSa} and Nefzi~\cite{N}, where the Hilbert transform and certain higher-order Riesz transforms for the [DHO,$\mathbb{Z}^1_2$] and $[{\rm DHO},\mathbb{Z}_2^d]$, respectively, were studied. We also get a weighted extension of the results by Forzani, Sasso and Scotto~\cite{FSS} for the heat semigroup maximal operator in the context of $[{\rm DHO},\mathbb{Z}_2^d]$. It is worth mentioning that recently unweighted $L^p$-boundedness of f\/irst order Riesz transforms and imaginary powers associated with the DHO and an arbitrary f\/inite group of ref\/lections was obtained by Amri~\cite{A} and Amri and Tayari~\cite{AT}, respectively. See~\cite{BRT,BT} for more results in the general DHO setting. Our present analysis is a natural, but by no means trivial, f\/irst step towards conjecturing and proving further results for the DHO and a general ref\/lection group, possibly with non-positive multiplicity functions admitted.

Our general strategy in proving $L^p$ mapping properties of the above mentioned operators is essentially the same as in our previous papers \cite{NoSt4,NoSt3,TZS2,TZS3}. Thus, using symmetries involved, we f\/irst reduce the analysis to a number of suitably def\/ined auxiliary operators related to a smaller measure metric space, which is actually a space of homogeneous type. Then to treat these auxiliary operators we apply the general vector-valued Calder\'on--Zygmund theory for spaces of homogeneous type. Here the most dif\/f\/icult step is to show the so-called standard estimates for the integral kernels involved. To achieve that, we employ the technique of kernel estimates that was inspired by Sasso's article~\cite{Sa} and then developed in the f\/irst two authors' paper \cite{NoSt2}, and later gradually ref\/ined by the f\/irst and third authors in~\cite{NoSz, TZS1,TZS2,TZS3}. Among all these references especially~\cite{NoSz} is relevant for our purposes since the tools established there allow to cover non-positive multiplicity functions. It is worth pointing out that there exists a variant of the Calder\'on--Zygmund theory suited to the general Dunkl setting with arbitrary f\/inite group of ref\/lections, see Amri and Sif\/i \cite{A,AS2,AS1}, but as stated it does not cover the case of a non-positive multiplicity function and does not provide weighted results.

\looseness=1 Another substantial aim of this paper is connected with the f\/irst two authors' papers \cite{NoSt1,NoSt5}. The latter article proposes a symmetrization procedure in a context of general discrete orthogonal expansions related to a second order dif\/ferential operator $L$, a `Laplacian'. This procedure, combined with a unif\/ied conjugacy scheme established in \cite{NoSt1} allows one to associate, via a~suitable embedding, a~dif\/ferential-dif\/ference `Laplacian' $\mathbb{L}$ with the initially given orthogonal system of eigenfunctions of $L$ so that the resulting extended conjugacy scheme has the natural classical shape. In particular, the related `partial derivatives' decomposing $\mathbb{L}$ are formally skew-adjoint in an appropriate~$L^2$ space and they commute with Riesz transforms and conjugate Poisson integrals. Thus the symmetrization procedure overcomes the main inconvenience of the theory postulated in \cite{NoSt1}, that is the lack of symmetry in the principal objects and relations resul\-ting in essential deviations of the theory from the classical shape. The price is, however, that the `Laplacian' $\mathbb{L}$ and the associated `partial derivatives' are not dif\/ferential, but dif\/ferential-dif\/ference operators. It was shown in~\cite{NoSt5} that the symmetrization is supported by a good $L^2$ theory. However, it seems to be practically impossible to develop the $L^p$ theory on the level of generality assumed in \cite{NoSt5}. Thus it is of interest and importance to look at the problem in concrete classical settings where proper tools and techniques are either known or can be ef\/fectively elaborated. Recently Langowski \cite{L1,L2,L3} verif\/ied that, in case of one-dimensional Jacobi trigonometric polynomial and function contexts, the symmetrization leads to an extended setting admitting a~good~$L^p$ theory. In the present paper we take the opportunity to investigate another, this time multi-dimensional, concrete realization of the symmetrization procedure and f\/ind out that it admits a good $L^p$ theory as well, giving further support for the theory in~\cite{NoSt5}. More precisely, we apply the real variant of the symmetrization procedure to the multi-dimensional Laguerre function setting of convolution type. This results in the Laguerre-symmetrized setting, which turns out to be closely related to the Laguerre--Dunkl context. Consequently, fundamental harmonic analysis operators can be analyzed by means of essentially the same strategy and technical tools. The outcome of our investigation is contained in Theorem~\ref{Sthm:main}.

The aforementioned framework of Laguerre function expansions of convolution type was widely studied from harmonic analysis perspective in the last decade or two; see, for instance, the authors' papers \cite{NoSt2,NoSt6,NoSt7,NoSz,TZS1,TZS2,TZS3} and references given there. The results obtained in this paper for the Laguerre--Dunkl and the Laguerre-symmetrized situations contribute also to this line of research. This is because, roughly speaking, the Laguerre setting can be recovered either from the Laguerre--Dunkl or the Laguerre-symmetrized context via a restriction to symmetric functions. Then our present results can be projected suitably to deliver new information about $L^p$ mapping properties of interesting variants of mixed $g$-functions, mixed Lusin area integrals and Riesz transforms (all of them of arbitrary orders) that were not investigated earlier in the Laguerre setting; see Theorem~\ref{thm:lag}. One also gets enhancements of existing results like, for instance, $L^p$-boundedness of f\/irst order Lusin area integrals valid for a complete range of the associated parameter of type; see Theorem~\ref{thm:lags}.

The paper is organized as follows. In Section~\ref{sec:prel} we introduce the notation and the three settings considered, that is the Laguerre, Laguerre--Dunkl and Laguerre-symmetrized situations. In Section~\ref{sec:main} we state the main results of the paper, Theorems~\ref{thm:main} and~\ref{Sthm:main}. Here we also introduce auxiliary Laguerre-type operators in the Laguerre--Dunkl and the Laguerre-symmetrized settings, which are related to a smaller space. Then we state Theorems~\ref{thm:main+} and~\ref{Sthm:main+} which allow us to reduce the proofs of Theorems~\ref{thm:main} and~\ref{Sthm:main} to a simpler situation
involving only the auxi\-liary operators. The ultimate reduction is due to the general Calder\'on--Zygmund theory, see Theorems~\ref{thm:CZ} and~\ref{Sthm:CZ}. We f\/inish this section by stating new results pertaining to the Laguerre setting, see Theorems~\ref{thm:lag} and~\ref{thm:lags}, and indicating further results in all the contexts consi\-de\-red. The most technical part of the paper, Section~\ref{sec:ker}, is devoted to the proofs of Theorems~\ref{thm:kerest} and~\ref{Sthm:kerest} that contain standard estimates for kernels associated with the above mentioned auxiliary Laguerre-type operators and thus deliver the missing link in the proofs of Theorems~\ref{thm:CZ} and~\ref{Sthm:CZ}. The paper f\/inishes with two appendices containing, respectively, a~minor auxiliary result and a table summarizing notation of various objects in the three contexts considered in this paper.

\newpage

\section{Preliminaries} \label{sec:prel}

It this section we introduce notation used throughout the paper and the three settings investigated. Notation of main objects in these settings is summarized in Table~\ref{table1}, which is located at the end of the paper for easy reference.

\subsection{Notation}

In the whole paper $d \ge 1$ and $\a \in (-1,\infty)^d$ will denote the dimension of the underlying space and the parameter of type appearing in all the contexts considered. These quantities should be thought of as f\/ixed from now on. We let $|\a| = \a_1+\dots + \a_d$ and point out that this sum may be negative. We denote by $\mu_{\a}$ the measure in $\R^d$ given by
\begin{gather*}
d\mu_{\a}(x) = \prod_{i=1}^d |x_i|^{2\a_i+1}\, dx_i.
\end{gather*}
The restriction of $\mu_{\a}$ to $\R^d_{+}=(0,\infty)^d$ will be written as $\mu_{\a}^{+}$.

Throughout the paper we use a fairly standard notation with essentially all symbols referring to the measure metric spaces $(\R^d,\mu_{\a},\|\cdot\|)$ and $(\R^d_{+},\mu_{\a}^{+},\|\cdot\|)$, where $\|\cdot\|$ stands for the Euclidean norm. In particular, by $\langle f,g \rangle_{d\mu_{\a}^{+}}$ we mean $\int_{\R^d_{+}} f(x)\overline{g(x)}\,d\mu_{\a}^{+}(x)$ whenever the integral makes sense. By $L^p(\R^d_{+},Ud\mu_{\a}^{+}) = L^p(Ud\mu_{\a}^{+})$ we understand the weighted $L^p(d\mu_{\a}^{+})$ space, $U$ being a~non-negative weight on $\R^d_{+}$. Further, for $1 \le p < \infty$ we write $A_p^{\a,+}$ for the Muckenhoupt class of $A_p$ weights connected with the space of homogeneous type $(\R^d_{+},\mu_{\a}^{+},\|\cdot\|)$. In an analogous way we interpret $\langle f,g\rangle_{d\mu_{\a}}$, $L^p(\R^d,Wd\mu_{\a}) = L^p(Wd\mu_{\a})$ and $A_p^{\a}$. Furthermore, by $C_0$ we denote the closed separable subspace of $L^{\infty}(\R_{+},dt)$ consisting of all continuous functions on $\mathbb{R}_{+}$ which have f\/inite limits as $t\to 0^{+}$ and vanish as $t\to \infty$.

Given $x,y \in \R^d_{+}$, $\beta \in \R^d$, $r>0$ and a multi-index $m \in \N^d$, being $\N=\{0,1,2,\ldots\}$, we denote:
\begin{gather*}
\boldsymbol{0} = (0,\ldots,0) \in \N^d, \qquad \boldsymbol{1} = (1,\ldots,1) \in \N^d, \\
|m| = m_1+\dots + m_d \qquad \textrm{(length of $m$)},\\
\|x\| = \big(x_1^2+\dots +x_d^2\big)^{1/2} \qquad \textrm{(Euclidean norm)},\\
B(x,r) = \big\{y \in \R^d_{+}\colon \|y-x\|< r\} \qquad \textrm{(open balls in $\R^d_{+}$)},\\
xy = (x_1 y_1,\ldots,x_d y_d), \\
x \vee y = (\max\{x_1,y_1\},\ldots,\max\{x_d,y_d\}), \\
x \wedge y = (\min\{x_1,y_1\},\ldots,\min\{x_d,y_d\}), \\
x^{\beta} = x_1^{\beta_1} \cdots x_d^{\beta_d},\\
x \le y \equiv x_i \le y_i, \qquad i = 1,\ldots,d, \\
\lfloor x \rfloor = (\max\{n \in \Z\colon n \le x_1\},\ldots,\max\{n \in \Z\colon n \le x_d\}) \qquad\textrm{(f\/loor function)}, \\
\overline{m} = (\overline{m}_1,\ldots, \overline{m}_d), \qquad	\overline{m}_i = m_i-2\lfloor m_i/2\rfloor = \chi_{\{m_i \; \textrm{is odd}\}}, \\
\partial_{x_i} = \partial\slash \partial x_i, \qquad i =1,\ldots,d \qquad \textrm{(classical partial derivatives)},\\
\partial_x^{m} = \partial_{x_1}^{m_1} \circ \cdots \circ \partial_{x_d}^{m_d} \qquad \textrm{(classical higher-order partial derivatives)}.
\end{gather*}
Moreover, if $x,y \in \R^d$ and $\beta \in \N^d$ we understand the objects $xy$, $x^{\beta}$, $\lfloor x \rfloor$ and the relation $x \le y$ in the same way as above whenever it makes sense.

{\allowdisplaybreaks Further, for $i= 1,\ldots, d$, $x,y \in \R^d$, $s \in (-1,1)^d$ and $\z \in (0,1)$, we introduce the following notation and abbreviations:
\begin{gather*}
e_i \equiv \textrm{$i$th coordinate vector in $\R^d$},\\
\sigma_i \equiv \textrm{ref\/lection with respect to the hyperplane $\{e_i\}^{\perp}$},\\
q_{\pm} = q_{\pm}(x,y,s) = \|x\|^2+\|y\|^2\pm 2\sum_{i=1}^d x_i y _i s_i,\\
\e(\z,q_{\pm}) = \exp\left(-\frac{1}{4\z}q_{+} - \frac{\z}4 q_{-}\right),\qquad \lo(\z) = \log\frac{1+\z}{1-\z}.
\end{gather*}}

We shall also use the following terminology. Given $\eta \in \Z_2^d = \{0,1\}^d$, we say that a function $f \colon \R^d \to \C$ is $\eta$-symmetric if for each $i=1,\ldots,d$, $f$ is either even or odd with respect to the $i$th coordinate according to whether $\eta_i=0$ or $\eta_i=1$, respectively. If $f$ is {$\boldsymbol{0}$}-symmetric, then we simply say that $f$ is symmetric or, alternatively, ref\/lection invariant. Furthermore, if there exists $\eta \in \Z_2^d$ such that $f$ is $\eta$-symmetric, then we denote by $f^{+}$ its restriction to $\R^d_{+}$. Finally, $f_{\eta}$ denotes the $\eta$-symmetric component of~$f$, i.e.,
\begin{gather*}
f = \sum_{\eta\in \{0,1\}^d} f_{\eta}, \qquad f_{\eta}(x) = \frac{1}{2^d} \sum_{\eps \in \{-1,1\}^d} \eps^{\eta}f(\eps x).
\end{gather*}
Conversely, if $f\colon \R^d_{+} \mapsto \C$, then by $f^{\eta}$ we mean the $\eta$-symmetric extension of $f$ to the whole $\R^d$, namely
\begin{gather*}
f^{\eta}(x) = \begin{cases}
								\eps^{\eta} f(\eps x), & \textrm{if $x \in (\R\setminus \{0\})^d$ and
										$\eps \in \{-1,1\}^d$ is such that $\eps x \in \R^d_{+}$},\\
								0, & \textrm{if $x \notin (\R\setminus \{0\})^d$}.
							\end{cases}
\end{gather*}

When writing estimates, we will frequently use the notation $X \lesssim Y$ to indicate that $X \le C Y$ with a positive constant $C$ independent of signif\/icant quantities. We shall write $X \simeq Y$ when simultaneously $X \lesssim Y$ and $Y \lesssim X$.

All the notation introduced in this section is essentially consistent with~\cite{NoSz}.

\subsection{Laguerre setting of convolution type}

The Laguerre functions of convolution type are given by
\begin{gather*}
\ell_k^{\a}(x) = c_k^{\a} \exp\big({-}\|x\|^2/2\big) \prod_{i=1}^{d} L_{k_i}^{\a_i}\big(x_i^2\big),
	\qquad k \in \N^d,
\end{gather*}
where $c_k^{\a}> 0$ are the normalizing constants, and $L_{k_i}^{\a_i}$ are the classical one-dimensional Laguerre polynomials. The system $\{\ell_k^{\a}\colon k \in \N^d\}$ is an orthonormal basis in $L^2(d\mu_{\a}^{+})$.

The $\ell_k^{\a}$ are eigenfunctions of the Bessel harmonic oscillator
\begin{gather*}
L_{\a} = -\Delta - \sum_{i=1}^d \frac{2\a_i+1}{x_i} \frac{\partial}{\partial x_i} + \|x\|^2
\end{gather*}
acting on $\R^d_{+}$. We have $L_{\a} \ell_k^{\a} = \lambda_{|k|}^{\a} \ell_k^{\a}$, where
\begin{gather*}
\lambda_n^{\a} = 4n + 2|\a| + 2d, \qquad n \ge 0.
\end{gather*}
We denote by the same symbol $L_{\a}$ the natural self-adjoint extension in $L^2(d\mu_{\a}^{+})$ whose spectral resolution is given by the $\ell_k^{\a}$.

Partial derivatives associated with $L_{\a}$ emerge from the decomposition
\begin{gather*}
L_{\a} = \lambda_0^{\a} + \sum_{i=1}^{d} \delta_i^{*} \delta_i,
\end{gather*}
where
\begin{gather*}
\delta_i = \frac{\partial}{\partial x_i} + x_i, \qquad \delta_i^{*} = -\frac{\partial}{\partial x_i} + x_i - \frac{2\a_i+1}{x_i},
\end{gather*}
$\delta_i^{*}$ being the formal adjoint of $\delta_i$ in $L^2(d\mu_{\a}^{+})$. Note that the action of $\delta_i$ on $\ell_k^{\a}$ is much simpler comparing to $\delta_i^{*}$, in fact we have (see \cite[p.~652]{NoSt2} or \cite[p.~694]{NoSt1})
\begin{gather*} 
\delta_i \ell_{k}^{\a}(x) = -2\sqrt{k_i} x_i \ell_{k-e_i}^{\a+e_i}(x), \qquad
\delta_i^{*} \ell_{k}^{\a}(x) = 2\sqrt{k_i} x_i \ell_{k-e_i}^{\a+e_i}(x) + \left(2x_i-\frac{2\a_i+1}{x_i}\right) \ell_k^{\a}(x);
\end{gather*}
here and elsewhere we use the convention that $\ell_k^{\a}\equiv 0$ if $k \notin \N^d$. Hence it is natural to consider $\delta_i$, $i=1,\ldots,d$, as the f\/irst order partial derivatives related to $L_{\a}$. This choice is further motivated by mapping properties of fundamental harmonic analysis operators involving derivatives, like Riesz--Laguerre transforms. On the other hand, the proper choice of higher-order derivatives is a more complicated matter. Taking into account a f\/ixed $i$th axis, one can iterate $\delta_i$ or interlace it with $\delta_i^*$. In fact, both possibilities are well motivated and of interest.

The heat semigroup $T_t^\a = \exp(-tL_{\a})$, $t >0$, has an integral representation in $L^2(d\mu_{\a}^{+})$ and the heat kernel is explicitly given by
\begin{gather*}
G_t^{\a}(x,y) = \frac{1}{(\sinh 2t)^d} \exp\left(\!{-}\frac{1}2 \coth(2t) \big( \|x\|^2 + \|y\|^2\big) \!\right)\!
	\prod_{i=1}^d (x_i y_i)^{-\a_i} I_{\a_i}\left( \frac{x_i y_i}{\sinh 2t }\right),\! \!\!\qquad t >0.
\end{gather*}
Here $I_{\a_i}$ denotes the modif\/ied Bessel function of the f\/irst kind of order $\a_i$. As a function on~$\R_{+}$, $I_{\a_i}$ is smooth and strictly positive. The following representation of~$G_t^{\a}(x,y)$, which is crucial in case $\a \notin [-1/2,\infty)^d$, was derived in~\cite{NoSz}:
\begin{gather} \label{IRL}
G_t^{\a}(x,y) = \sum_{\eps \in \{0,1\}^d} C_{\a,\eps} \left( \frac{1-\z^2}{2\z}\right)^{d+|\a|+2|\eps|}	(xy)^{2\eps} \int \e(\z,q_{\pm}) \Pi_{\a+\boldsymbol{1}+\eps}(ds),
\end{gather}
where $C_{\a,\eps} = [2(\a+\boldsymbol{1})]^{\boldsymbol{1}-\eps}$, $t$ and $\z$ are related by $\z = \tanh t$ or, equivalently,
\begin{gather*}
t = t(\z) = \frac{1}2 \log \frac{1+\z}{1-\z}, \qquad \z \in (0,1),
\end{gather*}
and $\Pi_{\nu}$ is the measure on $(-1,1)^d$ given by
\begin{gather*}
\Pi_{\nu}(ds) = \frac{1}{\pi^{d/2} 2^{|\nu|}} \prod_{i=1}^d \frac{\big(1-s_i^2\big)^{\nu_i-1/2}\, ds_i}{\Gamma(\nu_i+1/2)}, \qquad \nu \in (0,\infty)^d.
\end{gather*}
Here and elsewhere we omit writing the set of integration with respect to $\Pi_{\nu}$, which is always the cube~$(-1,1)^d$. Recall that the integrated expression $\e(\z,q_{\pm})$ depends implicitly also on~$x$,~$y$ and $s$. We remark that there is a more elementary representation of $G_t^{\a}(x,y)$ in the spirit of~\eqref{IRL}, but it is restricted to $\a\in [-1/2,\infty)^d$, see~\cite{NoSt2,NoSz}.

In the Laguerre setting objects like $\ell_k^{\a}$, $L_{\a}$, $\delta_i$, $\delta_i^{*}$, are considered on $\R_{+}^d$. Nevertheless, the same def\/ining formulas extend them naturally to the whole $\R^d$. In what follows we will use these extensions with the same notation and without further mention. An analogous remark pertains to the heat kernel $G_t^{\a}(x,y)$ and the above formulas, possibly with a limiting interpretation when some coordinates of~$x$ or~$y$ vanish.

\subsection{Laguerre--Dunkl setting}

This situation corresponds to the Dunkl harmonic oscillator in $\R^d$ and the associated group of ref\/lections isomorphic to $\mathbb{Z}_2^d$. The multi-parameter $\a$ represents the so-called multiplicity function, which is non-negative if and only if $\a \in [-1/2,\infty)^d$. For $\a = -\boldsymbol{1}/2$ the setting reduces to the context of the classical harmonic oscillator in $\R^d$.

The Laguerre--Dunkl functions are def\/ined on $\R^d$ by
\begin{gather*}
h_k^{\a}(x) = (-1)^{|\lfloor{k/2}\rfloor|} 2^{-d/2} x^{\overline{k}} \ell_{\lfloor{k/2}\rfloor}^{\a+\overline{k}}(x), \qquad k \in \N^d.
\end{gather*}
In the terminology of Dunkl theory the $h_k^{\a}$ are called generalized Hermite functions. The system $\{h_k^{\a} \colon k \in \N^d\}$ is an orthonormal basis in~$L^2(d\mu_{\a})$. We mention that $h_k^\a$ is $\eta$-symmetric if and only if $\eta = \overline{k}$. The associated Laguerre--Dunkl Laplacian is the dif\/ferential-dif\/ference operator given by
\begin{gather*}
\mathfrak{L}_{\a}f(x) = L_{\a}f(x) + \sum_{i=1}^d (\a_i+1/2) \frac{f(x)-f(\sigma_i x)}{x_i^2}.
\end{gather*}
We have $\mathfrak{L}_{\a} h_k^{\a} = \lambda_{|k|/2}^{\a} h_k^{\a}$. The natural in this context self-adjoint extension of $\mathfrak{L}_{\a}$ in $L^2(d\mu_{\a})$ will be denoted by the same symbol.

We have the symmetric decomposition
\begin{gather*}
\mathfrak{L}_{\a} = \frac{1}{2} \sum_{i=1}^d \big( \mathfrak{D}_i^{*}\mathfrak{D}_i
	+ \mathfrak{D}_i \mathfrak{D}_i^{*}\big),
\end{gather*}
where
\begin{gather*}
\mathfrak{D}_i = T_i^{\a} + x_i, \qquad \mathfrak{D}_i^{*} = -T_i^{\a} + x_i,
\end{gather*}
are the mutual formal adjoints in $L^2(d\mu_{\a})$, being
\begin{gather*}
T_i^{\a}f(x) = \frac{\partial}{\partial x_i} f(x) + (\a_i+1/2) \frac{f(x)-f(\sigma_i x)}{x_i}
\end{gather*}
(in the Dunkl theory $T_i^{\a}$, $i=1,\ldots,d$, are called Dunkl operators). For symmetry reasons, both~$\mathfrak{D}_i$ and~$\mathfrak{D}_i^{*}$ are the natural f\/irst order partial derivatives associated with $\mathfrak{L}_{\a}$. Their action on~$h_k^{\a}$ is (see \cite[p.~546]{NoSt3})
\begin{gather} \label{de_LD}
\mathfrak{D}_i h_k^{\a} = m(k_i,\a_i) h_{k-e_i}^{\a}, \qquad \mathfrak{D}_i^{*} h_k^{\a} = m(k_i+1,\a_i) h_{k+e_i}^{\a},
\end{gather}
where $m(k_i,\a_i) = \sqrt{2k_i + 2\overline{k_i}(2\a_i+1)}$; here and in other places we use the convention that $h_k^{\a} \equiv 0$ when $k \notin \N^d$.

Higher-order derivatives associated with $\mathfrak{L}_{\a}$ are formed by arbitrary f\/inite compositions of $\mathfrak{D}_i$ and $\mathfrak{D}_i^{*}$, $i = 1,\ldots,d$. For a given $n \in \N^d$ and a block multi-index $\om = (\om^1, \ldots, \om^d) \in \{-1,1\}^{n_1} \times \cdots \times \{-1,1\}^{n_d} = \{-1,1\}^{|n|}$ we denote
\begin{gather*}
\mathfrak{D}^{n,\om} = \mathfrak{D}_{d,n_d,\om^{d}} \circ \cdots \circ\mathfrak{D}_{1,n_1,\om^{1}},
\end{gather*}
where, for $i=1,\ldots,d$,
\begin{gather*}
\mathfrak{D}_{i,n_i,\om^i}
	= \big(\om_{n_i}^i T_i^{\a}+x_i\big) \circ \cdots \circ \big(\om_1^i T_i^{\a}+x_i\big).
\end{gather*}
By convention, $\mathfrak{D}_{i,0,\om^i} = \Id$. The action of $\mathfrak{D}^{n,\om}$ on $h_k^{\a}$ can be exactly described by means of~\eqref{de_LD},
but we will not need this. For our purposes we only need to notice that
\begin{gather} \label{A}
\mathfrak{D}^{n,\om} h_k^{\a} = \tau_{\om}^{\a}(k) h^{\a}_{k-\sum_{i=1}^d |\om^i|e_i},
\end{gather}
where $|\om^i| = \om_1^i + \dots + \om_{n_i}^i$ and for a f\/ixed $n \in \N^d$ the coef\/f\/icients satisfy
\begin{gather} \label{B}
0 \le \tau_{\om}^{\a}(k) \lesssim (|k|+1)^{|n|/2} \simeq \big(\lambda_{|k|/2}^{\a}\big)^{|n|/2}, \qquad k \in \N^d, \qquad \om \in \{-1,1\}^{|n|}.
\end{gather}
It is worth pointing out that $\tau_{\om}^{\a}(k)$ vanishes if and only if there exist $1\le i \le d$ and $1 \le j \le n_i$ such that $k_i-(\om_1^i+\dots +\om_j^i) < 0$. Finally, observe that if $f$ is $\eta$-symmetric then $\mathfrak{D}^{n,\om}f$ is $(\overline{\eta+n})$-symmetric. In particular, $\mathfrak{D}^{n,\om}h_k^{\a}$ is $(\overline{k+n})$-symmetric.

The Laguerre--Dunkl heat semigroup $\mathfrak{T}_t^\a = \exp(-t\mathfrak{L_{\a}})$, $t >0$, has an integral representation in $L^2(d\mu_{\a})$, and the integral kernel can be represented as, see, e.g., \cite[equation~(3)]{NoSt7} for the one-dimensional case,
\begin{gather} \label{GD}
\mathfrak{G}_t^{\a}(x,y) = \frac{1}{2^d} \sum_{\eta \in \{0,1\}^d} (xy)^{\eta} G_t^{\a+\eta}(x,y), \qquad t>0.
\end{gather}
This and \eqref{IRL} leads to the important representation
\begin{gather*} 
\mathfrak{G}_t^{\a}(x,y) = \frac{1}{2^d} \sum_{\eps,\eta \in \{0,1\}^d} C_{\a+\eta,\eps}
	\left( \frac{1-\z^2}{2\z}\right)^{d+|\a|+|\eta|+2|\eps|}(xy)^{\eta+2\eps} \int \e(\z,q_{\pm}) \Pi_{\a+\eta+\boldsymbol{1}+\eps}(ds).
\end{gather*}
Note that the sum in \eqref{GD} contains some cancellations since all the terms except one may take negative values when certain coordinates of $x$ and $y$ have opposite signs. Nevertheless, the ker\-nel~$\mathfrak{G}_t^{\a}(x,y)$ is strictly positive when $\a \in [-1/2,\infty)^d$. On the other hand, it may
be shown, see~\cite{NoSt7}, that the kernel takes both positive and negative values if~$\a$ does not satisfy the latter condition. Observe the correlation between positivities of the heat kernel and the multiplicity function.

\subsection{Laguerre-symmetrized setting}

This framework arises by applying the real variant of the symmetrization procedure proposed in~\cite{NoSt5} to the situation of Laguerre function expansions of convolution type, see \cite[Example~5.2]{NoSt5}.

The symmetrized system $\{\Phi_k^{\a}\colon k \in \N^d\}$ is an orthonormal basis in $L^2(d\mu_{\a})$. The $\Phi_k^{\a}$ coincide with $h_k^{\a}$ up to signs (in particular, $\Phi_k^{\a}$ is $\eta$-symmetric if and only if $\eta = \overline{k}$). More precisely,
\begin{gather*}
\Phi_k^{\a}(x) = (-1)^{|\lfloor k/2\rfloor|} h_k^{\a}(x) = 2^{-d/2} x^{\overline{k}} \ell_{\lfloor k/2 \rfloor}^{\a+\overline{k}}(x), \qquad k \in \N^d.
\end{gather*}
However, the Laguerre-symmetrized Laplacian $\mathbb{L}_{\a}$ dif\/fers slightly from the Laguerre--Dunkl Lap\-la\-cian. We have
\begin{gather*}
\mathbb{L}_{\a}f(x) = \mathfrak{L}_{\a}f(x) + \sum_{i=1}^d \big[ f(x)-f(\sigma_i x)\big].
\end{gather*}
Consequently, $\mathbb{L}_{\a}\Phi_k^{\a} = \lambda_{|\lfloor (k+\boldsymbol{1})/2\rfloor|}^{\a} \Phi_k^{\a}$. Accordingly, we consider the natural in this context self-adjoint extension of $\mathbb{L}_{\a}$ in $L^2(d\mu_{\a})$.

Partial derivatives corresponding to $\mathbb{L}_{\a}$ are def\/ined by
\begin{gather*}
\mathbb{D}_if(x) = \frac{\partial}{\partial x_i} f(x) + x_i f(\sigma_i x) +	\frac{\a_i+1/2}{x_i} \big[ f(x) - f(\sigma_i x)\big].
\end{gather*}
This is motivated by the following two facts. Each $\mathbb{D}_i$ is formally skew-adjoint in $L^2(d\mu_{\a})$ and~$\mathbb{L}_{\a}$ decomposes simply as
\begin{gather*}
\mathbb{L}_{\a} = \lambda_{0}^{\a} - \sum_{i=1}^d \mathbb{D}_i^2.
\end{gather*}
Observe that there is no ambiguity in def\/ining higher-order derivatives in this context. Note that (see \cite[Lemma~4.2]{NoSt5})
\begin{gather} \label{de_LS}
\mathbb{D}_i \Phi_k^{\a} = (-1)^{k_i+1}2\sqrt{\lfloor (k_i+1)/2 \rfloor}\, \Phi^{\a}_{k-(-1)^{k_i}e_i}
\end{gather}
with the convention that $\Phi_k^{\a} \equiv 0$ if $k \notin \N^d$. Iteration of~\eqref{de_LS} easily leads to the description of the action of $\mathbb{D}^n = \mathbb{D}^{n_d}_d \circ \cdots \circ \mathbb{D}^{n_1}_1 $ on $\Phi_k^{\a}$, see \cite[Corollary~4.3]{NoSt5},
\begin{gather} \label{C}
\mathbb{D}^n \Phi_k^{\a} = \rho_n(k) \Phi^{\a}_{k-(-\boldsymbol{1})^k \overline{n}},
\end{gather}
where the coef\/f\/icients satisfy
\begin{gather} \label{D}
|\rho_n(k)| \lesssim (|k|+1)^{|n|/2} \simeq \big( \lambda_{|\lfloor (k+\boldsymbol{1})/2\rfloor|}^{\a}\big)^{|n|/2}, \qquad k \in \N^d.
\end{gather}
Notice that $\rho_n(k)$ vanishes if and only if there exist $1\le i \le d$ such that $k_i=0$ and $n_i>0$. Further, it is easy to see that if $f$ is $\eta$-symmetric, then $\mathbb{D}^{n} f$ is $(\overline{\eta + n})$-symmetric. In particular, $\mathbb{D}^{n} \Phi_k^{\a}$ is $(\overline{k + n})$-symmetric.

The Laguerre-symmetrized heat semigroup $\mathbb{T}_t^\a = \exp(-t\mathbb{L_{\a}})$, $t >0$, is expressed in $L^2(d\mu_{\a})$ by integrating against the kernel
\begin{gather*}
\mathbb{G}_t^{\a}(x,y) = \frac{1}{2^d} \sum_{\eta \in \{0,1\}^d} \exp\big({-}2|\eta|t\big) (xy)^{\eta} 	G_t^{\a+\eta}(x,y), \qquad t >0.
\end{gather*}

It should be pointed out that the Laguerre-symmetrized and the Laguerre--Dunkl settings reduce essentially to the Laguerre convolution setting after restriction to symmetric (ref\/lection invariant) functions.

\section{Main results} \label{sec:main}

In this section we present the main results of the paper. Of prior importance are Sections~\ref{ssec:L-D} and~\ref{ssec:L-s} covering the Laguerre--Dunkl and the Laguerre-symmetrized settings, respectively. As it was already mentioned, the general strategy of proving~$L^p$ mapping properties of operators considered in both settings relies on reducing the analysis to a number of suitably def\/ined auxiliary operators related to a smaller measure metric space. This strategy has its roots in the papers \cite{NoSt4, NoSt3} and was then successfully applied in \cite{CaSz,L1,L3,N,TZS2,TZS3}. Accordingly, the proofs of the two main
results, Theorems~\ref{thm:main} and~\ref{Sthm:main}, are reduced to two auxiliary results, Theorems~\ref{thm:main+} and~\ref{Sthm:main+}. The proofs of the two latter results fall under the well-known scope of the general (vector-valued) Calder\'on--Zygmund theory for spaces of homogeneous type. Analogous approach based on the Calder\'on--Zygmund theory in similar situations can be found in \cite{BCN,CaSz,L1,L3,N,NoSt2,NoSt4,NoSt3,NoSz,StTo2,StTo1,TZS1,TZS2,TZS3}, among many others. For the reader's convenience we recall the main ingredients of this theory in what follows.

Some complementary results are contained in Sections~\ref{ssec:lag} and~\ref{ssec:fur}. More precisely, in Section~\ref{ssec:lag} we state some new results concerning Riesz transforms and square functions in the Laguerre setting. Most of them, see Theorem~\ref{thm:lag}, are consequences of the Laguerre--Dunkl and Laguerre-symmetrized results, but some not, cf.\ Theorem~\ref{thm:lags}. Finally, in Section~\ref{ssec:fur} we comment on further results in all the three frameworks considered that can be obtained by means of the techniques presented in this paper. This concerns, in particular, the Poisson semigroup based operators such as the maximal operator, Laplace multipliers, $g$-functions and Lusin area integrals; see Theorem~\ref{thm:mainP}.

\subsection{Laguerre--Dunkl setting} \label{ssec:L-D}

The main objects of our study in this subsection are the following operators, which are def\/ined initially in~$L^2(d\mu_{\a})$.
\begin{description}\itemsep=0pt
\item[{\rm (L-D.I)}] The Laguerre--Dunkl heat semigroup maximal operator
\begin{gather*}
\mathfrak{T}_{*}^{\alpha}f = \big\| \mathfrak{T}_t^{\alpha}f\big\|_{L^{\infty}(\R_+,dt)}.
\end{gather*}
\item[{\rm (L-D.II)}] Riesz--Laguerre--Dunkl transforms of order $|n|>0$ and type $\om$
\begin{gather*}
\mathfrak{R}_{n,\om}^{\alpha}f = \sum_{k \in \N^d}
\big( \lambda_{|k|/2}^{\a} \big)^{-|n|\slash 2}
	\langle f , h_k^{\a} \rangle_{d\mu_{\alpha}}\, \mathfrak{D}^{n,\om} h_k^{\a},
\end{gather*}
where $n \in \N^d \setminus \{\boldsymbol{0}\}$ and $\om \in \{-1,1\}^{|n|}$.
\item[{\rm (L-D.III)}] Multipliers of Laplace and Laplace--Stieltjes transform types
\begin{gather*}
\mathfrak{M}^{\alpha}_{\mathfrak{m}} f = \sum_{k \in \N^d} \mathfrak{m}( \lambda_{|k|/2}^{\a} )
		\langle f , h_k^{\a} \rangle_{d\mu_{\alpha}} h_k^{\a},
\end{gather*}
where either $\mathfrak{m}(z) = z\int_0^{\infty} e^{-tz} \psi(t)\, dt$ with $\psi \in L^{\infty}(\R_+,dt)$ or $\mathfrak{m}(z) = \int_{\R_+} e^{-tz} \, d\nu (t)$ with~$\nu$ being a signed or complex Borel measure on $\R_+$, with its total variation $|\nu|$ satisfying
\begin{gather}\label{assum}
\int_{\R_+} e^{-t \lambda_{0}^{\a} } \, d|\nu|(t) < \infty.
\end{gather}
\item[{\rm (L-D.IV)}] Littlewood--Paley--Stein type mixed $g$-functions
\begin{gather*}
\mathfrak{g}^{\alpha}_{n,m,\om}(f) = \big\| \partial_t^m \mathfrak{D}^{n,\om}
	\mathfrak{T}_t^{\alpha}f \big\|_{L^2(\R_+,t^{|n|+2m-1}dt)},
\end{gather*}
where $n \in \N^d$, $m \in \N$ are such that $|n|+m>0$, and $\om \in \{-1,1\}^{|n|}$.
\item[{\rm (L-D.V)}] Mixed Lusin area integrals
\begin{gather*}
\mathfrak{S}^{\a}_{n,m,\om}(f)(x)= \left( \int_{A(x)} t^{|n| + 2m -1} \big|\partial_{t}^m \mathfrak{D}^{n,\om}
	\mathfrak{T}_t^{\alpha}f (z)\big|^{2}\frac{d\mu_{\a}(z) \, dt}{V_{\sqrt{t}}^{\a}(x)} \right)^{1\slash 2},
\end{gather*}
where $n \in \N^d$, $m \in \N$ are such that $|n|+m>0$, and $\om \in \{-1,1\}^{|n|}$. Further, here $A(x)$ is the parabolic cone with vertex at $x$,
\begin{gather}\label{def:A}
A(x)=(x,0)+A,\qquad A=\big\{(z,t)\in \RR \times \R_+ \colon |z|<\sqrt{t}\big\}
\end{gather}
(note that the exact aperture of this cone is meaningless for our developments) and $V_{t}^{\a}(x)$ is the~$\mu_\a$ measure of the cube centered at~$x$ and of side lengths~$2t$. More precisely,
\begin{gather}\label{def:V}
V_{t}^{\a}(x)=\prod_{i=1}^{d}V_{t}^{\a_{i}}(x_i),\qquad V_{t}^{\a_i}(x_i)=\mu_{\a_i}\big((x_{i}-t,x_{i}+t)\big),\qquad x\in\RR,\quad t>0.
\end{gather}
\end{description}
The series def\/ining $\mathfrak{R}_{n,\om}^{\alpha}$ and $\mathfrak{M}^{\alpha}_{\mathfrak{m}}$ converge in $L^2(d\mu_\a)$ and produce $L^2(d\mu_\a)$-bounded ope\-ra\-tors. This is obvious in case of $\mathfrak{M}^{\alpha}_{\mathfrak{m}}$ since the values of $\mathfrak{m}$ stay bounded. For $\mathfrak{R}_{n,\om}^{\alpha}$ it follows from \eqref{A} and \eqref{B}, via Parseval's identity. Moreover, the formulas def\/ining $\mathfrak{T}^{\alpha}_{*}f$, $\mathfrak{g}^{\alpha}_{n,m,\om}(f)$ and $\mathfrak{S}^{\a}_{n,m,\om}(f)$, understood in a pointwise sense, are valid (the series/integral def\/ining $\mathfrak{T}_t^{\alpha}f(x)$ converges and produces a smooth function of $(x,t)\in \RR \times\mathbb{R}_{+}$) for $f \in L^{p}(Wd\mu_\a)$, $W\in A_{p}^{\a}$, $1 \le p<\infty$; see Proposition~\ref{pro:conv} in Appendix~I.

Our main result in the Laguerre--Dunkl situation reads as follows.
\begin{Theorem}\label{thm:main}\sloppy
Assume that $\a \in(-1,\infty)^d$ and $W$ is a weight in $\RR$ invariant under the reflec\-tions $\sigma_1, \ldots ,\sigma_d$. Then the Riesz--Laguerre--Dunkl transforms {\rm (L-D.II)} and the multi\-pliers of Laplace and Lap\-lace--Stieltjes transform types {\rm (L-D.III)} extend to bounded linear ope\-ra\-tors on $L^{p}(Wd\mu_\a)$, $W\in A_{p}^{\a}$, $1<p<\infty$, and from $L^1 (Wd\mu_\a)$ to weak $L^1 (Wd\mu_\a)$, $W\in A_{1}^{\a}$. Further\-more, the Laguerre--Dunkl heat semigroup maximal operator {\rm (L-D.I)}, the mixed $g$-functions \mbox{{\rm (L-D.IV)}} and the mixed Lusin area integrals {\rm (L-D.V)} are bounded sublinear operators on $L^{p}(Wd\mu_\a)$, $W\in A_{p}^{\a}$, $1<p<\infty$, and from $L^1 (Wd\mu_\a)$ to weak $L^1 (Wd\mu_\a)$, $W\in A_{1}^{\a}$.
\end{Theorem}

Notice that for symmetric weights the condition $W\in A_{p}^{\a}$ is equivalent to saying that \mbox{$W^{+}\in A_{p}^{\a,+}$} (recall that $W^+$ is the restriction of $W$ to $\RRR$). We now relate in detail Theorem~\ref{thm:main} to earlier results pertaining to the operators (L-D.I)--(L-D.V).

The Laguerre--Dunkl maximal operator (L-D.I) was considered in \cite{NoSt3} by the f\/irst two \mbox{authors} and claimed to be bounded on $L^p(W d\mu_{\a})$, $W \in A_p^{\a}$, $1<p<\infty$, and from $L^1(W d\mu_{\a})$ to weak $L^1(W d\mu_{\a})$, $W \in A_1^{\a}$, provided that $\a \in [-1/2,\infty)^d$; see \cite[Theorem~3.1]{NoSt3}. Unfortunately, the argument justifying this in \cite[p.~545]{NoSt3} works only for symmetric weights. Thus in \cite[Theorem~3.1]{NoSt3} it should be assumed that $W$ is ref\/lection invariant. An unweighted version of this result, but with all $\a \in (-1,\infty)^d$ admitted, is due to Forzani, Sasso and Scotto \cite[Theorem~1.9(a)]{FSS}. Theorem~\ref{thm:main} delivers new weighted results in case $\a \in (-1,\infty)^d\setminus [-1/2,\infty)^d$. We remark that the part of Theorem~\ref{thm:main} related to (L-D.I) could be obtained in a more elementary way, essentially by the above mentioned arguments from~\cite{NoSt3}, where instead of Soni's inequality one proceeds with the aid of standard Bessel function asymptotics getting the bound (cf.\ \cite[p.~545, line~9]{NoSt3})
\begin{gather*}
| \mathfrak{G}_t^{\a}(x,y) | \lesssim G_t^{\a}(x,y), \qquad x,y \in \RR, \qquad t > 0,
\end{gather*}
for all $\a \in (-1,\infty)^d$ and then appealing to \cite[Corollary~4.2]{NoSz} rather than \cite[Theorem~2.1]{NoSt2}, since the former covers all $\a$ as above. Nevertheless, the proof of Theorem~\ref{thm:main} shows, roughly speaking, that $\mathfrak{T}^{\a}_{*}$ can be ``decomposed'' into Calder\'on--Zygmund operators, which potentially provides an approach to investigating more subtle mapping properties of $\mathfrak{T}^{\a}_{*}$ involving, for instance, Hardy and BMO spaces.

The Riesz--Laguerre--Dunkl transforms of order $1$ def\/ined by means of $\mathfrak{D}_i$ (whereas their counterparts based on $\mathfrak{D}_i^{*}$ were not taken into account) were studied by the f\/irst two authors in~\cite{NoSt3} under the restriction $\a \in [-1/2,\infty)^d$ and claimed there to be bounded on $L^p(W d\mu_{\a})$, $W \in A_p^{\a}$, $1<p<\infty$, and from $L^1(W d\mu_{\a})$ to weak $L^1(W d\mu_{\a})$, $W \in A_1^{\a}$; see \cite[Theorem~4.3]{NoSt3}. Unfortunately, the arguments given in~\cite{NoSt3} prove this only for ref\/lection invariant weights $W$. We take this opportunity to point out the corrections needed in \cite[Section~4]{NoSt3}. The weak association from \cite[Proposition~4.1]{NoSt3} should be stated and proved for the operators $R_j^{\a,\varepsilon}$ def\/ined by the component kernels \cite[equation~(5.1)]{NoSt3}. The standard estimates stated in \cite[Theorem~4.2]{NoSt3} should be replaced by analogous estimates for the component kernels just mentioned, see \cite[Lemma~5.1]{NoSt3}. Finally, \cite[Theorem~4.3]{NoSt3} should say that the operators $R_j^{\a,\varepsilon}$ (instead of~$R_j^{\a}$) are Calder\'on--Zygmund and it should be assumed that the weights~$W$ are ref\/lection inva\-riant.
Coming back to our bibliographical account, we next mention a paper by Ben Salem and Samaali~\cite{BSa} where, in dimension~$1$, Riesz--Laguerre--Dunkl (or rather Hilbert--Laguerre--Dunkl) transforms of order~$1$, based both on $\mathfrak{D}_1$ and $\mathfrak{D}_1^{*}$, were studied with the restriction $\a \ge -1/2$. For such operators unweighted $L^p$ boundedness was obtained in \cite[Theorem~5]{BSa}. The authors also claim principal value integral representations for the Hilbert--Laguerre--Dunkl transforms in \cite[Theorem~4]{BSa}, but the corresponding reasoning does not appear to us to be correct. More recently, assuming $\a \in [-1/2,\infty)^d$, Nefzi~\cite{N} considered higher-order Riesz--Laguerre--Dunkl transforms def\/ined either via~$\mathfrak{D}_i$ or via $\mathfrak{D}_i^*$ (but not both of them simultaneously), with at most~$1$ dif\/ferentiation in each coordinate direction allowed (this, in particular,
excludes orders higher than the dimension). The main result of~\cite{N} is parallel to that from~\cite{NoSt3} and, unfortunately, it inherits the error discussed above. Thus the result itself and the arguments justifying it should be corrected according to the above given indications. Having this background, we see that Theorem~\ref{thm:main} generalizes the earlier results concerning the Riesz--Laguerre--Dunkl transforms in several directions: by admitting all $\a \in (-1,\infty)^d$, all orders of the transforms, and more general form of derivatives entering the def\/inition, which are arbitrary compositions of~$\mathfrak{D}_i$ and~$\mathfrak{D}_i^*$. It is interesting to observe that with the trivial choice of the multiplicity function (the case $\alpha=-\boldsymbol{1}/2$) Theorem~\ref{thm:main} brings new results even for the higher-order Riesz--Hermite transforms investigated by Harboure, de Rosa, Segovia and Torrea~\cite{HRST}, and the second-named author and Torrea~\cite{StTo1}. More precisely, in~\cite{StTo1} the authors prove the mapping properties from Theorem~\ref{thm:main} for the Riesz--Hermite transforms without imposing the symmetry of weights, but with less general derivatives def\/ining these operators (composition of~$\mathfrak{D}_i$ and $\mathfrak{D}_j^{*}$ is not allowed, even if their action is related to dif\/ferent coordinate directions).
On the other hand, in~\cite{HRST} the derivatives are as general as in our present result, but neither weights are allowed nor the case $p=1$ is treated there (the main objective of~\cite{HRST} are dimension free~$L^p$ estimates). Finally, we note that recently Riesz transforms associated with the DHO and an arbitrary group of ref\/lections were studied by Amri~\cite{A} and Boggarapu and Thangavelu~\cite{BT}, in both cases with only non-negative multiplicity functions admitted. More precisely, in~\cite{A} unweighted $L^p$-boundedness, $1< p < \infty$, and weak type $(1,1)$ for Riesz--Dunkl transforms of order~$1$ (def\/ined by means of counterparts of~$\mathfrak{D}_i$, but not~$\mathfrak{D}_i^*$) were obtained. In \cite{BT} the authors prove mixed norm estimates (weighted $L^{p,2}$-boundedness, $1<p<\infty$) for Riesz--Dunkl transforms of order $1$ def\/ined via the counterparts of both $\mathfrak{D}_i$ and $\mathfrak{D}_i^{*}$. Our Theorem~\ref{thm:main} suggests that the results of~\cite{A,BT} can be substantially generalized.

A particular instance of the Laplace transform type multipliers (L-D.III), imaginary powers of the DHO, was investigated in \cite{NoSt4} by the f\/irst two authors. It was shown that these operators are bounded on $L^p(W d\mu_{\a})$, $W\in A_p^{\a}$, $1<p<\infty$, and from $L^1(W d\mu_{\a})$ to weak $L^1(W d\mu_{\a})$, $W \in A_1^{\a}$, where $W$ are ref\/lection invariant weights and $\a \in [-1/2,\infty)^d$; see \cite[Theorem~1 and Remark 1]{NoSt4}. This result was later generalized by the third-named author to arbitrary Laplace and also Laplace--Stieltjes transform type multipliers \cite[Theorem~2.2]{TZS2}, under the assumption $\a \in [-1/2,\infty)^d$. Theorem~\ref{thm:main} removes this restriction by admitting all $\a \in (-1,\infty)^d$. It is worth mentioning that the Laplace--Stieltjes transform type multipliers~(L-D.III) cover, as special cases, potential operators associated with $[{\rm DHO},\Z_2^d]$. The latter operators were studied by the f\/irst two authors in~\cite{NoSt6,NoSt7}. In \cite{NoSt6} the emphasis was put on two weight $L^p-L^q$ inequalities with radial power weights involved, under the assumption $\a \in [-1/2,\infty)^d$. Nevertheless, in case of $p=q>1$ and equal weights, Theorem~\ref{thm:main} implies further results for the potential operators by covering $\a \notin [-1/2,\infty)^d$ and by allowing weights not admitted in~\cite{NoSt6}. On the other hand, the aim of~\cite{NoSt7} was to obtain sharp unweighted $L^p-L^q$ estimates for the potential operators in dimension~$1$, but with all $\a>-1$ admitted. Again, Theorem~\ref{thm:main} brings here some new results in the case $p=q>1$ by allowing weights. Recently Wr\'obel \cite{W} proved a Marcinkiewicz type multivariate spectral multiplier theorem in the Laguerre--Dunkl context. Assuming $\a \in [-1/2,\infty)^d$ he infers that multipliers much more general than (L-D.III) are $L^p$-bounded, $1<p<\infty$, with no weights admitted; see \cite[Theorem~4.1]{W}. Finally, also recently, imaginary powers of the DHO related to an arbitrary group of ref\/lections were investigated by Amri and Tayari~\cite{AT}. They proved that for non-negative multiplicity functions the imaginary powers are bounded on~$L^p$, $1<p<\infty$, and from $L^1$ to weak $L^1$ (with no weights allowed). Clearly, the part of Theorem~\ref{thm:main} related to (L-D.III) indicates a natural generalization of this result.

Passing to the Littlewood--Paley--Stein type $g$-functions (L-D.IV), essentially we can only invoke the third author's paper~\cite{TZS3}, where the f\/irst order $g$-functions were considered and the special case of Theorem~\ref{thm:main} related to $\mathfrak{g}_{n,m,\om}^{\a}$, determined by the restrictions $|n|+m=1$ and $\a \in [-1/2,\infty)^d$, was obtained; see \cite[Theorem~2.1]{TZS3}. The special case $\a = -\boldsymbol{1}/2$ (trivial multiplicity function) deserves additional comments in this context. It seems that so far only the vertical $g$-functions $\mathfrak{g}_{\boldsymbol{0},m,\om}^{-\boldsymbol{1}/2}$ were investigated and mapping properties as in Theorem~\ref{thm:main} were obtained, though without requiring the symmetry of weights; see Thangavelu \cite[Section~4.1]{Th} and \cite[Section~2]{StTo2} by the second-named author and Torrea. So the part of Theorem~\ref{thm:main} related to $g$-functions provides a meaningful generalization of existing results even in the framework of the classical harmonic oscillator.

Considering the Lusin area type integrals (L-D.V), again essentially we can only appeal to the third author's paper \cite{TZS3} which delivers the special case of the part of Theorem~\ref{thm:main} pertaining to $\mathfrak{S}_{n,m,\om}^{\a}$ given by the restrictions $|n|+m=1$ and $\a \in [-1/2,\infty)^d$; see \cite[Theorem~2.1]{TZS3}. No more seems to have been done in the classical harmonic oscillator context (the case $\a=-\boldsymbol{1}/2$). We only mention that in this special case~\cite{TZS3} extended the results of Betancor, Molina and Rodr\'i{}guez-Mesa~\cite{BMR}, where one-dimensional vertical Lusin area integrals were studied ($n=0$, $m=1$), but in a slightly more general form emerging from involving $L^r$ norms, $r \ge 2$, rather than $L^2$ norms in the def\/inition.

This somewhat lengthy account reveals importance and strength of Theorem~\ref{thm:main}.
We will now proceed with proving this result.
The proof of Theorem~\ref{thm:main} can be reduced to showing analogous mapping properties for certain,
suitably def\/ined, auxiliary Laguerre-type operators emerging
from those introduced above and related to the smaller space $(\RRR,\mu_\a^+,\|\cdot\|)$.
To begin with, for each $\eta \in \{ 0,1 \}^d$ we consider an auxiliary semigroup acting initially on
$L^2(d\mu_\a^+)$ and given by
\begin{gather*}
\mathfrak{T}_t^{\a,\eta,+}f = \big( \mathfrak{T}_t^{\a} f^\eta \big)^+= \sum_{k \in \N^d, \, \overline{k} = \eta}
e^{-t \lambda_{|k|/2}^{\a}} \langle f^{\eta} , h_k^{\a} \rangle_{d\mu_{\alpha}}\, \big( h_k^{\a} \big)^+, \qquad t > 0;
\end{gather*}
observe that $\mathfrak{T}_t^{\a} f^\eta$ is $\eta$-symmetric, which means that the quantity $\mathfrak{T}_t^{\a,\eta,+}f$ is well def\/ined. It is not hard to check that these auxiliary semigroups have the integral representations
\begin{gather} \nonumber
\mathfrak{T}_t^{\a,\eta,+} f (x) = \int_{\RRR} \mathfrak{G}_t^{\a,\eta,+} (x,y) f(y) \, d\mu_\a^+(y), \qquad x\in \RRR, \qquad t>0,\\ \label{def:aheatD}
\mathfrak{G}_t^{\a,\eta,+} (x,y) = (xy)^\eta G_t^{\a + \eta} (x,y), \qquad x,y \in \RRR, \qquad t>0.
\end{gather}
Further, these series/integral formulas coincide and provide a good def\/inition of $\mathfrak{T}_t^{\a,\eta,+}$ on weighted $L^p$ spaces for a large class of weights and produce always smooth functions of $(x,t) \in \RRR \times \mathbb{R}_+$, see Proposition~\ref{pro:conv}. Note that choosing $\eta_0=\boldsymbol{0}$ we have $\mathfrak{T}_t^{\a,\eta_0,+} =T^{\a}_t$.

For $\eta \in \{ 0,1 \}^d$, $n \in \N^d$ and $\om \in \{ -1,1\}^{|n|}$ we denote
\begin{gather*}
\delta_{\eta, n, \om} = \delta_{d,\eta_d, n_d, \om^d} \circ \cdots \circ \delta_{1,\eta_1, n_1, \om^1},
\end{gather*}
where for each $i \in \{1, \ldots, d \}$ we put
\begin{gather*}
\delta_{i,\eta_i, n_i, \om^i} = \big( \om^i_{n_i} \partial_{i, \overline{\eta_i + n_i - 1} } + x_i \big)
\circ \cdots \circ \big( \om^i_2 \partial_{i, \overline{\eta_i + 1} } + x_i \big) \circ \big( \om^i_1 \partial_{i, \overline{\eta_i} } + x_i \big)
\end{gather*}
(by convention, $\delta_{i,\eta_i, 0, \om^i} = \Id$) and
\begin{gather} \label{pp}
\partial_{i, \eta_i } = \partial_{x_i} + \eta_i \frac{2\a_i + 1}{x_i}.
\end{gather}
Notice that the derivatives $\partial_{i, \eta_i }$ and $\delta_{\eta, n, \om}$ correspond to the action of $T_i^\a$ and $\mathfrak{D}^{n, \om}$ on $\eta$-symmetric functions, respectively. To be more precise, if $f$ is $\eta$-symmetric, then $T_i^\a f = \partial_{i, \eta_i } f$ and $\mathfrak{D}^{n, \om} f = \delta_{\eta, n, \om} f$. Moreover, we may also think that $\partial_{i, \eta_i }$ and $\delta_{\eta, n, \om}$ act on functions def\/ined on the restricted space $\RRR$.

Now we are ready to introduce the auxiliary Laguerre-type operators, which are def\/ined initially in $L^2(d\mu_\a^+)$. For each $\eta \in \{0,1\}^d$ we def\/ine the following objects.
\begin{description}\itemsep=0pt
\item[{\rm (L-t.I)}] The Laguerre-type heat semigroup maximal operator
\begin{gather*}
\mathfrak{T}_{*}^{\alpha,\eta, +} f =
\big\| \mathfrak{T}_t^{\alpha,\eta, +} f\big\|_{L^{\infty}(\R_+,dt)}.
\end{gather*}
\item[{\rm (L-t.II)}] Laguerre-type Riesz transforms of order $|n|>0$ and type $\om$
\begin{gather*}
\mathfrak{R}_{n,\om}^{\alpha,\eta, +} f = \sum_{ k \in \N^d, \, \overline{k} = \eta}
\big( \lambda_{|k|/2}^{\a} \big)^{-|n|\slash 2}
	\langle f^\eta , h_k^{\a} \rangle_{d\mu_{\alpha}}\,
\big( \delta_{\eta, n, \om} h_k^{\a} \big)^+,
\end{gather*}
where $n \in \N^d \setminus \{\boldsymbol{0}\}$
and $\om \in \{-1,1\}^{|n|}$.
Observe that if $\overline{k} = \eta$, then
$\delta_{\eta, n, \om} h_k^{\a} = \mathfrak{D}^{n,\om} h_k^{\a}$ is $(\overline{k + n})$-symmetric and hence
the quantity $\big( \delta_{\eta, n, \om} h_k^{\a} \big)^+$ is well def\/ined.
\item[{\rm (L-t.III)}] Multipliers of Laplace and Laplace--Stieltjes transform types
\begin{gather*}
\mathfrak{M}^{\alpha, \eta, +}_{\mathfrak{m}} f =
\sum_{ k \in \N^d, \, \overline{k} = \eta}
\mathfrak{m}( \lambda_{|k|/2}^{\a} )
		\langle f^\eta , h_k^{\a} \rangle_{d\mu_{\alpha}}\,
\big( h_k^{\a} \big)^+,
\end{gather*}
where $\mathfrak{m}$ is as in (L-D.III).
\item[{\rm (L-t.IV)}] Littlewood--Paley--Stein type mixed $g$-functions
\begin{gather*}
\mathfrak{g}^{\alpha,\eta,+}_{n,m,\om}(f) = \big\| \partial_t^m \delta_{\eta, n, \om} \mathfrak{T}_t^{\alpha,\eta, +}f \big\|_{L^2(\R_+,t^{|n|+2m-1}dt)},
\end{gather*}
where $n \in \N^d$, $m \in \N$ are such that $|n|+m>0$, and $\om \in \{-1,1\}^{|n|}$.
\item[{\rm (L-t.V)}] Mixed Lusin area integrals
\begin{gather*}
\mathfrak{S}^{\a,\eta,+}_{n,m,\om}(f)(x)= \left( \int_{A(x)\cap \R^{d+1}_+} t^{|n| + 2m -1} \big|\partial_{t}^m
\delta_{\eta, n, \om} \mathfrak{T}_t^{\alpha,\eta, +} f (z) \big|^{2}\frac{d\mu_{\a}^+(z) \, dt} {V_{\sqrt{t}}^{\a,+}(x)} \right)^{1\slash 2},
\end{gather*}
where $n \in \N^d$, $m \in \N$ are such that $|n|+m>0$, and $\om \in \{-1,1\}^{|n|}$. Further, $A(x)$ is the parabolic cone with vertex at~$x$, see~\eqref{def:A}. Here $V_{t}^{\a,+}(x)$ is the $\mu^+_{\a}$ measure of the cube centered at~$x$ and of side lengths~$2t$, restricted to $\RRR$. More precisely,
\begin{gather}\label{def:V+}
V_{t}^{\a,+}(x) =\prod_{i=1}^{d}V_{t}^{\a_{i},+}(x_i), \qquad x\in\RRR,\qquad t>0,\\ \nonumber
V_{t}^{\a_i,+}(x_i) = \mu_{\a_{i}}^{+}\big((x_{i}-t,x_{i}+t)\cap\mathbb{R}_{+}\big),\qquad x_i > 0,\qquad t>0.
\end{gather}
\end{description}
Notice that the Laguerre-type Lusin area integrals can be written as
\begin{gather*}
\mathfrak{S}^{\a,\eta,+}_{n,m,\om}(f)(x)=\big\| \partial_t^m \delta_{\eta, n, \om} \mathfrak{T}_t^{\alpha,\eta, +} f (x+z) \sqrt{\Xi_{\a}(x,z,t)}
 \chi_{\{x+z \in \RRR \}}\big\|_{L^2(A,t^{ |n| + 2m - 1} dzdt)},
\end{gather*}
where the function $\Xi_{\a}$ is given by
\begin{gather}\label{def:Xi}
\Xi_{\a}(x,z,t)=\prod_{i=1}^{d} \frac{(x_{i}+z_{i})^{2\a_{i} + 1}}{V_{\sqrt{t}}^{\a_{i},+}(x_{i})},
\qquad x\in\RRR,\qquad z\in\RR,\qquad x+z\in\RRR, \qquad t > 0.
\end{gather}
The series def\/ining $\mathfrak{R}_{n,\om}^{\alpha,\eta,+}$ and $\mathfrak{M}^{\alpha,\eta,+}_{\mathfrak{m}}$ converge in $L^2(d\mu_\a^+)$ and produce $L^2(d\mu_\a^+)$-bounded operators. This follows from the analogous properties of $\mathfrak{R}_{n,\om}^{\alpha}$ and $\mathfrak{M}^{\alpha}_{\mathfrak{m}}$, for symmetry reasons. Further, the formulas def\/ining $\mathfrak{T}^{\alpha,\eta,+}_{*}f$, $\mathfrak{g}^{\alpha,\eta,+}_{n,m,\om}(f)$ and $\mathfrak{S}^{\a,\eta,+}_{n,m,\om}(f)$ make sense in a pointwise way for general functions $f$, see Proposition~\ref{pro:conv}.

Arguments similar to those given in \cite[p.~6]{NoSt4} and \cite[pp.~1522--1524]{TZS3} allow us to reduce the proof of Theorem~\ref{thm:main} to showing the following.

\begin{Theorem}\label{thm:main+}
Assume that $\a \in(-1,\infty)^d$ and $\eta \in \{ 0, 1\}^d$. Then the Laguerre-type operators {\rm (L-t.II)} and {\rm (L-t.III)} extend to bounded linear operators on $L^{p}(Ud\mu^+_{\a})$, $U \in A_{p}^{\a,+}$, $1<p<\infty$, and from $L^{1}(Ud\mu^+_{\a})$ to weak $L^{1}(Ud\mu^+_{\a})$, $U \in A_{1}^{\a,+}$.
Furthermore, the sublinear operators~{\rm (L-t.I)}, {\rm (L-t.IV)} and {\rm (L-t.V)} are bounded on $L^{p}(Ud\mu^+_{\a})$, $U \in A_{p}^{\a,+}$, $1<p<\infty$, and from $L^1(Ud\mu^+_{\a})$ to weak $L^1(Ud\mu^+_{\a})$, $U \in A_{1}^{\a,+}$.
\end{Theorem}

To prove Theorem~\ref{thm:main+} we will use the general Calder\'on--Zygmund theory. In fact, we will show that the Laguerre-type operators (L-t.I)--(L-t.V) are (vector-valued) Calder\'on--Zygmund operators in the sense of the space of homogeneous type $(\RRR,\mu_\a^+,\|\cdot\|)$. Then, in particular, the mapping properties claimed in Theorem~\ref{thm:main+} will follow from the general theory and arguments similar to those mentioned for instance in the proof of \cite[Corollary~2.5]{TZS1}. To treat the Lusin area integrals we shall need a slightly more general def\/inition of the standard kernel, or rather standard estimates, than the one used in the papers~\cite{NoSt4,NoSt3,TZS2}. More precisely, we will allow slightly weaker smoothness estimates as indicated below, see for instance~\cite{CaSz,TZS3}.

Let $\mathbb{B}$ be a Banach space and let $K(x,y)$ be a kernel def\/ined on $\RRR \times \RRR \backslash \{(x,y)\colon x=y\}$ and taking values in~$\mathbb{B}$. We say that $K(x,y)$ is a standard kernel in the sense of the space of homogeneous type $(\RRR, \mu^+_{\a},\|\cdot\|)$ if it satisf\/ies the growth estimate
\begin{gather} \label{gr}
\|K(x,y)\|_{\mathbb{B}} \lesssim \frac{1}{\mu^+_{\a}(B(x,\|x-y\|))}, \qquad x \neq y,
\end{gather}
and the smoothness estimates
\begin{gather}
\| K(x,y)-K(x',y)\|_{\mathbb{B}} \lesssim \left(\frac{\|x-x'\|}{\|x-y\|} \right)^{\gamma} \frac{1}{\mu^+_{\a}(B(x,\|x-y\|))},\qquad \|x-y\|>2\|x-x'\|, \label{sm1}\\
\| K(x,y)-K(x,y')\|_{\mathbb{B}} \lesssim \left(\frac{\|y-y'\|}{\|x-y\|} \right)^{\gamma} \frac{1}{\mu^+_{\a}(B(x,\|x-y\|))}, \qquad \|x-y\|>2\|y-y'\|, \label{sm2}\!\!\!
\end{gather}
for some f\/ixed $\gamma > 0$. Notice that the bounds \eqref{sm1} and \eqref{sm2} imply analogous estimates with any $0<\gamma'<\gamma$ instead of~$\gamma$.
Further, observe that in these formulas, the ball $B(x,\|y-x\|)$ can be replaced by $B(y,\|x-y\|)$, in view of the doubling property of $\mu^+_{\a}$. Furthermore, when $K(x,y)$ is scalar-valued (i.e., $\mathbb{B}=\mathbb{C}$) and $\gamma = 1$, the dif\/ference bounds~\eqref{sm1} and~\eqref{sm2} are implied by the more convenient gradient estimate
\begin{gather} \label{grad}
\|\nabla_{\! x,y} K(x,y)\| \lesssim \frac{1}{\|x-y\|\mu^+_{\a}(B(x,\|x-y\|))}, \qquad x \neq y.
\end{gather}
Similar reduction holds also in the vector-valued situations we consider. Here, however, we will also use~\eqref{sm1} and~\eqref{sm2} with $\gamma <1$ and thus it is more convenient to verify the smoothness estimates rather than \eqref{grad}.

A linear operator $T$ assigning to each $f\in L^2(d\mu^+_{\a})$ a measurable $\B$-valued function~$Tf$ on~$\RRR$ is a (vector-valued) Calder\'on--Zygmund operator in the sense of the space $(\RRR,\mu^+_{\a},\|\cdot\|)$ if
\begin{itemize}\itemsep=0pt
 \item[(i)] $T$ is bounded from $L^2(d\mu^+_{\a})$ to $L^2_{\B}(d\mu^+_{\a})$,
 \item[(ii)] there exists a standard $\B$-valued kernel $K(x,y)$ such that
\begin{gather*}
Tf(x)=\int_{\RRR} K(x,y) f(y) \, d\mu^+_{\a}(y),\qquad \textrm{a.a.}\,\,\, x\notin \supp f,
\end{gather*}
for every $f \in L_c^{\infty}(\RRR)$,
where $L_c^{\infty}(\RRR)$ is the subspace of $L^{\infty}(\RRR)$ of bounded measurable functions
with compact supports.
\end{itemize}
Here integration of $\mathbb{B}$-valued functions is understood in Bochner's sense, and $L^2_{\mathbb{B}}(d\mu^+_{\a})$ is the Bochner--Lebesgue space of all $\mathbb{B}$-valued $\mu^+_{\a}$-square integrable functions on $\RRR$.

Classical theory of Calder\'on--Zygmund operators, see, e.g., \cite[Chapter~6]{Ch},
\cite[Chapter~5]{DUO} or \cite[Chapter~4]{G}, is nowadays a standard tool in analysis.
Moreover, it is well known that a large part of this theory remains valid,
with appropriate adjustments, when the underlying space is of homogeneous type and the associated kernels
are vector-valued, see for instance \cite{RRT} and \cite{RT}.

The following result, together with the arguments discussed above, implies Theorem~\ref{thm:main+} and thus also Theorem~\ref{thm:main}.

\begin{Theorem}\label{thm:CZ}
Assume that $\a \in (-1,\infty)^{d}$ and $\eta\in \{ 0 , 1 \}^d$. The Laguerre-type Riesz transforms~{\rm (L-t.II)} and the multipliers of Laplace and Laplace--Stieltjes transform types {\rm (L-t.III)} are scalar-valued Calder\'on--Zygmund operators in the sense of the space $(\RRR,\mu_{\a}^+,\|\cdot\|)$.
Furthermore, the Laguerre-type heat semigroup maximal operator {\rm (L-t.I)}, the mixed $g$-functions \mbox{{\rm (L-t.IV)}} and the mixed Lusin area integrals {\rm (L-t.V)} can be viewed as vector-valued Calder\'on--Zygmund operators in the sense of $(\RRR,\mu_{\a}^+,\|\cdot\|)$ associated with the Banach spaces $\mathbb{B}=C_0$, $\mathbb{B}=L^2(\R_+,t^{|n|+2m-1}dt)$ and $\mathbb{B}=L^2(A,t^{|n|+2m-1}dzdt)$, respectively.
\end{Theorem}

Proving Theorem~\ref{thm:CZ} splits naturally into showing the following three results (Propositions~\ref{pro:L2b} and~\ref{pro:kerassoc}, and Theorem~\ref{thm:kerest}).

\begin{Proposition}\label{pro:L2b}
Let $\a\in(-1,\infty)^d$ and $\eta\in \{ 0 , 1 \}^d$. Then the Laguerre-type operators from Theorem~{\rm \ref{thm:CZ}} are bounded on $L^2(d\mu_{\a}^{+})$.
\end{Proposition}

\begin{proof}The $L^2(d\mu_\a^+)$-boundedness of $\mathfrak{R}_{n,\om}^{\a,\eta,+}$ and $\mathfrak{M}^{\alpha,\eta,+}_{\mathfrak{m}}$ is already justif\/ied, see the comment preceding the statement of Theorem~\ref{thm:main+}.

Considering $\mathfrak{T}_*^{\a,\eta,+}$, its $L^2(d\mu_{\a}^+)$-boundedness follows from the $L^2(d\mu_{\a})$-boundedness of $\mathfrak{T}_*^{\a}$ (see the comments following Theorem~\ref{thm:main}) via restricting its action to $\eta$-symmetric functions. Alternatively, one can argue more directly, similarly as it was done in the Bessel--Dunkl setting \cite[p.~953]{CaSz}. Observe that, see~\eqref{def:aheatD},
\begin{gather*}
\mathfrak{T}_{*}^{\a,\eta,+}f(x)=x^\eta T_{*}^{\a+\eta}\big({y^{-\eta}}f\big)(x), \qquad x\in\RRR,
\end{gather*}
where $T_*^\a f=\|T_t^\a f\|_{L^\infty(\R_+,dt)}$ is the Laguerre heat semigroup maximal operator. Since $T_{*}^{\a}$ is bounded on $L^2(d\mu_\a^+)$ (cf.\ \cite[Theorem~4.1]{NoSz} and also references given there), we obtain
\begin{gather*}
\big\|\mathfrak{T}_{*}^{\a,\eta,+}f \big\|_{L^2(d\mu_\a^+)} =\big\|T_{*}^{\a+\eta}\big( {y^{-\eta}}f\big)\big\|_{L^2(d\mu_{\a+\eta}^+)} \lesssim
\big\| {y^{-\eta}}f\big\|_{L^2(d\mu_{\a+\eta}^+)} = \|f\|_{L^2(d\mu_\a^+)}.
\end{gather*}

Passing to $\mathfrak{g}^{\a,\eta,+}_{n,m,\om}$, we f\/irst note that the special cases $|n|+m=1$ are contained in \cite[Proposition~2.4]{TZS3}. Actually, that result is stated under the assumption $\a \in [-1/2,\infty)^d$, nevertheless the argument given there is valid for all $\a \in (-1,\infty)^d$. To show the $L^2(d\mu_{\a}^+)$-boundedness of~$\mathfrak{g}^{\a,\eta,+}_{n,m,\om}$ in the general case it is enough to verify the $L^2(d\mu_{\a})$-boundedness of $\mathfrak{g}^{\a}_{n,m,\om}$, since then the desired property will follow via restricting to $\eta$-symmetric functions. Dif\/ferentiating the series def\/ining $\mathfrak{T}_t^{\a}f$ (this is legitimate, see the proof of Proposition~\ref{pro:conv}) and using \eqref{A} we get
\begin{gather*}
\partial_t^m \mathfrak{D}^{n,\om} \mathfrak{T}_t^{\a}f = \sum_{k \in \N^d} 	\big( {-} \lambda^{\a}_{|k|/2} \big)^m \tau_{\om}^{\a}(k) e^{-t\lambda^{\a}_{|k|/2}} \langle f, h_k^{\a}\rangle_{d\mu_{\a}} h^{\a}_{k-\sum_{i=1}^d |\om^i|e_i}.
\end{gather*}
Now, changing the order of integration and then using Parseval's identity and \eqref{B}, we arrive at the bound
\begin{gather*}
\big\| \mathfrak{g}^{\a}_{n,m,\om}(f)\big\|_{L^2(d\mu_{\a})}^2 \lesssim 	\int_0^{\infty} \left( \sum_{k \in \N^d} \big(\lambda_{|k|/2}^{\a}\big)^{|n|+2m} e^{-2t\lambda_{|k|/2}^{\a}} \big| \langle f, h_k^{\a}\rangle_{d\mu_{\a}}\big|^2 \right) t^{|n|+2m-1}\, dt.
\end{gather*}
Changing the order of integration and summation, evaluating the integral and then using once again Parseval's identity leads directly to the $L^2$ bound for $\mathfrak{g}^{\a}_{n,m,\om}$.

Finally, the $L^2(d\mu_{\a}^+)$-boundedness of $\mathfrak{S}_{n,m,\om}^{\a,\eta,+}$ is a~consequence of the same property for $\mathfrak{g}_{n,m,\om}^{\a,\eta,+}$ which is already justif\/ied. Indeed, with the aid of Lemma~\ref{lem:intLusin}(a) below one easily verif\/ies that
\begin{gather*}
\big\| \mathfrak{S}^{\a,\eta,+}_{n,m,\om} (f) \big\|_{L^2(d\mu_\a^+)}
	\simeq \big\| \mathfrak{g}^{\a,\eta,+}_{n,m,\om} (f) \big\|_{L^2(d\mu_\a^+)},
\qquad f \in L^2(d\mu_\a^+).
\end{gather*}
This f\/inishes the proof.
\end{proof}

Formal computations and the results from papers \cite{NoSt4,NoSt3,TZS2,TZS3} suggest that the Laguerre-type operators are associated with the following kernels related to appropriate Banach spaces~$\mathbb{B}$.
\begin{description}\itemsep=0pt
\item[{\rm (L-t.I)}] The kernel associated with the Laguerre-type heat semigroup maximal operator,
\begin{gather*}
\mathfrak{U}^{\alpha, \eta, +}(x,y) = \big\{\mathfrak{G}_t^{\alpha,\eta,+}(x,y)\big\}_{t>0},
\qquad \mathbb{B}= C_0 \subset L^{\infty}(\R_+,dt).
\end{gather*}
\item[{\rm (L-t.II)}] The kernels associated with the Laguerre-type Riesz transforms,
\begin{gather*}
\mathfrak{R}_{n,\om}^{\alpha,\eta,+}(x,y) = \frac{1}{\Gamma(|n|\slash 2)}
\int_0^{\infty} \delta_{\eta, n, \om, x} \mathfrak{G}_t^{\alpha, \eta, +}(x,y)
	t^{|n|\slash 2 -1}\, dt, \qquad \mathbb{B}=\mathbb{C},
\end{gather*}
where $n \in \N^d \setminus \{\boldsymbol{0}\}$ and $\om \in \{-1,1\}^{|n|}$.
\item[{\rm (L-t.IIIa)}] The kernels associated with the Laplace transform type multipliers,
\begin{gather*}
\mathfrak{K}^{\alpha, \eta, +}_{\psi}(x,y) = - \int_0^{\infty} \psi(t) \partial_t \mathfrak{G}_t^{\alpha, \eta, +}(x,y)\, dt,
		\qquad \mathbb{B}=\mathbb{C},
\end{gather*}
where $\psi \in L^{\infty}(\R_+,dt)$.
\item[{\rm (L-t.IIIb)}] The kernels associated with the Laplace--Stieltjes transform type multipliers,
\begin{gather*}
\mathfrak{K}^{\alpha, \eta, +}_{\nu}(x,y) = \int_{\R_+} \mathfrak{G}_t^{\alpha, \eta, +}(x,y)\, d\nu(t), \qquad \mathbb{B}=\mathbb{C},
\end{gather*}
where $\nu$ is a signed or complex Borel measure on $\R_+$ with the total variation $|\nu|$ satis\-fying~\eqref{assum}.
\item[{\rm (L-t.IV)}] The kernels associated with the mixed $g$-functions,
\begin{gather*}
\mathfrak{H}^{\alpha, \eta, +}_{n,m,\om}(x,y) = \big\{ \partial_t^m \delta_{\eta, n, \om, x}
\mathfrak{G}_t^{\alpha, \eta, +}(x,y) \big\}_{t>0}, \qquad \mathbb{B} = L^2(\R_+,t^{|n|+2m-1}dt),
\end{gather*}
where $n \in \N^d$, $m \in \N$ are such that $|n|+m>0$, and $\om \in \{-1,1\}^{|n|}$.
\item[{\rm (L-t.V)}] The kernels associated with the mixed Lusin area integrals
\begin{gather*}
\mathfrak{S}^{\alpha, \eta, +}_{n,m,\om}(x,y)
= \left\{ \partial_t^m \delta_{\eta, n, \om, \mathbf{x}} \mathfrak{G}_{t}^{\a,\eta,+}(\mathbf{x},y)
\Big|_{\mathbf{x} = x+z} \sqrt{\Xi_{\a}(x,z,t)}
\,\chi_{\{x+z\in\RRR\}} \right\}_{(z,t) \in A}
\end{gather*}
with $\mathbb{B} = L^2(A,t^{|n|+2m-1}dzdt)$, where $n \in \N^d$, $m \in \N$ are such that $|n|+m>0$, and $\om \in \{-1,1\}^{|n|}$.
\end{description}

The next result shows that the associations are indeed true in the Calder\'on--Zygmund theory sense.

\begin{Proposition}\label{pro:kerassoc}
Assume that $\a \in (-1, \infty)^d$ and $\eta \in \{ 0 , 1 \}^d$. Then the Laguerre-type operators {\rm (L-t.I)}--\emph{(L-t.V)} are associated, in the Calder\'on--Zygmund theory sense, with the corresponding kernels just listed.
\end{Proposition}

\begin{proof}
The reasoning is fairly standard. In the cases of $\mathfrak{U}^{\a,\eta,+}$ and $\mathfrak{R}_{n,\om}^{\a,\eta,+}$ we can proceed as in \cite{NoSt2,NoSt3} since the arguments given there are actually valid for all $\a\in(-1,\infty)^d$ provided that the same is true about the standard estimates. Similarly, in the cases of~$\mathfrak{K}_{\psi}^{\a,\eta,+}$ and $\mathfrak{K}_{\nu}^{\a,\eta,+}$ we can proceed as in~\cite{TZS2} since an analogous remark applies.
Finally, in the cases of~$\mathfrak{H}_{n,m,\om}^{\a,\eta,+}$ and~$\mathfrak{S}_{n,m,\om}^{\a,\eta,+}$ we can proceed as in~\cite{TZS1,TZS3} since, again, the same remark is in force. To be precise, in~\cite{TZS3} only some special cases of the present operators are covered but the arguments used there apply also in the more general situation of mixed $g$-functions and mixed Lusin area integrals considered in this paper. For readers' convenience we now sketch the proof of the association in case of the $g$-functions. Further details and the other cases are left to the reader.

Let $\mathbb{B} = L^2(\R_+,t^{|n|+2m-1}dt)$. By density arguments it is enough to show that
\begin{gather*}
 \Big\langle\big\{ \partial_t^m \delta_{\eta, n, \om}
\mathfrak{T}_t^{\alpha,\eta, +}f \big\}_{t>0},h\Big\rangle_{L_{\mathbb{B}}^2(d\mu_\a^+) }\\
\qquad{} =
\bigg\langle\int_{\RRR}\big\{ \partial_t^m \delta_{\eta, n, \om, x}
\mathfrak{G}_t^{\alpha, \eta, +}(x,y) \big\}_{t>0}
f(y)\,d\mu_{\a}^+ (y),h\bigg\rangle_{L_{\mathbb{B}}^2(d\mu_\a^+) }
\end{gather*}
for every $f\in C^{\infty}_c(\RRR)$ and $h(x,t)=h_{1}(x)h_{2}(t)$ such that $h_{1}\in C^{\infty}_c(\RRR)$, $h_{2}\in C^{\infty}_c(\mathbb{R}_{+})$ and $\supp f\cap \supp h_{1}=\varnothing$ (notice that the linear span of functions $h$ of this form is dense in $L_{\mathbb{B}}^2\big((\supp f)^{C},d\mu_{\a}^+\big)$). We f\/irst deal with the left-hand side of the desired identity. Using the $L^2(d\mu_\a^+)$-boundedness of $\mathfrak{g}^{\alpha,\eta,+}_{n,m,\om}$ we may change the order of integration and obtain
\begin{gather*}
\int_{0}^{\infty} t^{|n|+2m-1} \overline{h_{2}(t)} \int_{\RRR} \partial_t^m \delta_{\eta, n, \om} \mathfrak{T}_t^{\alpha,\eta, +}f (x)
\overline{h_{1}(x)} \, d\mu_{\a}^+ (x)\,dt.
\end{gather*}

On the other hand, using Fubini's theorem (its application is legitimate in view of the growth condition for the kernel $\mathfrak{H}^{\alpha, \eta, +}_{n,m,\om}(x,y)$) and changing the order of integration, we see that the right-hand side in question equals
\begin{gather*}
\int_{0}^{\infty} t^{|n|+2m-1} \overline{h_{2}(t)}\int_{\RRR} \int_{\RRR}
\partial_t^m \delta_{\eta, n, \om, x} \mathfrak{G}_t^{\alpha, \eta, +}(x,y) f(y)
\,d\mu_{\a}^+(y) \overline{h_{1}(x)} \,d\mu_{\a}^+ (x) \,dt.
\end{gather*}
Therefore, to f\/inish the reasoning, it suf\/f\/ices to verify that
\begin{gather*}
\partial_t^m \delta_{\eta, n, \om} \mathfrak{T}_t^{\alpha,\eta, +}f (x)=
\int_{\RRR} \partial_t^m \delta_{\eta, n, \om, x} \mathfrak{G}_t^{\alpha, \eta, +}(x,y) f(y)
\,d\mu_{\a}^+(y), \qquad x \in \RRR, \qquad t > 0,
\end{gather*}
for each $f\in C^{\infty}_c(\RRR)$. This, however, can be done by using the dominated convergence theorem and the estimates obtained in Lemma~\ref{lem:heatEST} below.
\end{proof}

Finally, we state the central technical result of our approach.
\begin{Theorem}\label{thm:kerest}
Let $\a \in (-1, \infty)^d$ and $\eta \in \{ 0 , 1 \}^d$. Then the kernels {\rm (L-t.I)--(L-t.V)} listed above satisfy the standard estimates with the relevant Banach spaces $\mathbb{B}$. More precisely, the kernels \mbox{{\rm (L-t.I)--(L-t.IV)}} satisfy the smoothness conditions with $\gamma = 1$, and the kernel {\rm (L-t.V)} satisfies~\eqref{sm1} and~\eqref{sm2} with any $\gamma \in (0,1/2]$ such that $\gamma < \min\limits_{1 \le i \le d} (\a_i + 1)$.
\end{Theorem}

The proof of Theorem~\ref{thm:kerest}, which is the most technical part of the paper, is located in Section~\ref{sec:ker}.

\subsection{Laguerre-symmetrized setting} \label{ssec:L-s}

The main objects of our interest in this subsection are the following operators, which are def\/ined initially in $L^2(d\mu_{\a})$.
\begin{description}\itemsep=0pt
\item[{\rm (L-s.I)}] The Laguerre-symmetrized heat semigroup maximal operator
\begin{gather*}
\mathbb{T}_{*}^{\alpha}f = \big\| \mathbb{T}_t^{\alpha}f\big\|_{L^{\infty}(\R_+,dt)}.
\end{gather*}
\item[{\rm (L-s.II)}] Riesz--Laguerre-symmetrized transforms of order $|n|>0$
\begin{gather*}
\mathbb{R}_{n}^{\alpha}f = \sum_{k \in \N^d}\big( \lambda_{|\lfloor (k+\boldsymbol{1})/2 \rfloor |}^{\a} \big)^{-|n|\slash 2}
	\langle f , \Phi_k^{\a} \rangle_{d\mu_{\alpha}}\, \mathbb{D}^{n} \Phi_k^{\a},
\end{gather*}
where $n \in \N^d \setminus \{\boldsymbol{0}\}$.
\item[{\rm (L-s.III)}] Multipliers of Laplace and Laplace--Stieltjes transform types
\begin{gather*}
\mathbb{M}^{\alpha}_{\mathfrak{m}} f = \sum_{k \in \N^d}\mathfrak{m}( \lambda_{|\lfloor (k+\boldsymbol{1})/2 \rfloor |}^{\a} )
		\langle f , \Phi_k^{\a} \rangle_{d\mu_{\alpha}}\, \Phi_k^{\a},
\end{gather*}
where $\mathfrak{m}$ is as in (L-D.III).
\item[{\rm (L-s.IV)}] Littlewood--Paley--Stein type mixed $g$-functions
\begin{gather*}
\mathbb{g}^{\alpha}_{n,m}(f) = \big\| \partial_t^m \mathbb{D}^{n}\mathbb{T}_t^{\alpha}f \big\|_{L^2(\R_+,t^{|n|+2m-1}dt)},
\end{gather*}
where $n \in \N^d$, $m \in \N$ are such that $|n|+m>0$.
\item[{\rm (L-s.V)}] Mixed Lusin area integrals
\begin{gather*}
\mathbb{S}^{\a}_{n,m}(f)(x)=\left( \int_{A(x)} t^{|n| + 2m -1} \big|\partial_{t}^m \mathbb{D}^{n}
\mathbb{T}_t^{\alpha}f (z)\big|^{2}\frac{d\mu_{\a}(z) \, dt}
{V_{\sqrt{t}}^{\a}(x)} \right)^{1\slash 2},
\end{gather*}
where $n \in \N^d$ and $m \in \N$ are such that $|n|+m>0$, and $A(x)$, $V_{t}^{\a}(x)$ are def\/ined in~\eqref{def:A} and~\eqref{def:V}, respectively.
\end{description}
The series def\/ining $\mathbb{R}_{n}^{\alpha}$ and $\mathbb{M}^{\alpha}_{\mathfrak{m}}$ converge in $L^2(d\mu_\a)$ and produce $L^2(d\mu_\a)$-bounded operators. This is immediate in case of $\mathbb{M}^{\alpha}_{\mathfrak{m}}$ since the values of~$\mathfrak{m}$ stay bounded. For $\mathbb{R}_{n}^{\alpha}$ it follows from~\eqref{C} and~\eqref{D}. Moreover, the formulas def\/ining $\mathbb{T}^{\alpha}_{*}f$, $\mathbb{g}^{\alpha}_{n,m}(f)$ and $\mathbb{S}^{\a}_{n,m}(f)$, understood in a~pointwise way, are valid (the series/integral def\/ining $\mathbb{T}_t^{\alpha}f(x)$ converges and produces a smooth function of $(x,t)\in \RR \times\mathbb{R}_{+}$) for $f \in L^{p}(Wd\mu_\a)$, $W\in A_{p}^{\a}$, $1 \le p<\infty$; see Proposition~\ref{pro:conv}.

Our main result in the Laguerre-symmetrized framework reads as follows.
\begin{Theorem}\label{Sthm:main}
Assume that $\a \in(-1,\infty)^d$ and $W$ is a weight on $\RR$ invariant under the reflections $\sigma_1, \ldots ,\sigma_d$. Then the Riesz--Laguerre-symmetrized transforms \emph{(L-s.II)} and the multipliers of Laplace and Laplace--Stieltjes transform types \emph{(L-s.III)} extend to bounded linear operators on $L^{p}(Wd\mu_\a)$, $W\in A_{p}^{\a}$, $1<p<\infty$, and from $L^1 (Wd\mu_\a)$ to weak $L^1 (Wd\mu_\a)$, $W\in A_{1}^{\a}$. Furthermore, the Laguerre-symmetrized heat semigroup maximal operator \emph{(L-s.I)}, the mixed $g$-functions \emph{(L-s.IV)} and the mixed Lusin area integrals \emph{(L-s.V)} are bounded sublinear operators on $L^{p}(Wd\mu_\a)$, $W\in A_{p}^{\a}$, $1<p<\infty$, and from $L^1 (Wd\mu_\a)$ to weak $L^1 (Wd\mu_\a)$, $W\in A_{1}^{\a}$.
\end{Theorem}

It is worth mentioning that an analogue of Theorem~\ref{Sthm:main} in the one-dimensional framework of Jacobi trigonometric polynomial expansions was proved recently by Langowski \cite{L1,L3}, though without including Lusin area integrals; see \cite[Theorem~2.1]{L1} and \cite[Theorem~3.1]{L3}. Apart from that, no other symmetrized settings seem to have been studied earlier from a similar perspective.

The proof of Theorem~\ref{Sthm:main} can be reduced to showing analogous mapping properties for certain, suitably def\/ined, auxiliary Laguerre-type operators emerging from those introduced above and related to the smaller space $(\RRR,\mu_\a^+,\|\cdot\|)$. To proceed, for each $\eta \in \{ 0,1 \}^d$ we consider an auxiliary semigroup of operators acting initially on $L^2(d\mu_\a^+)$ and given by
\begin{gather*}
\mathbb{T}_t^{\a,\eta,+}f = \big( \mathbb{T}_t^{\a} f^\eta \big)^+=
\sum_{k \in \N^d, \, \overline{k} = \eta}
e^{-t \lambda_{|\lfloor (k+\boldsymbol{1})/2 \rfloor |}^{\a} } \langle f^{\eta} , \Phi_k^{\a}
\rangle_{d\mu_{\alpha}}\, \big( \Phi_k^{\a} \big)^+,
\qquad t > 0;
\end{gather*}
observe that $\mathbb{T}_t^{\a} f^\eta$ is $\eta$-symmetric, which means that the quantity $\mathbb{T}_t^{\a,\eta,+}f$ is well def\/ined. It is straightforward to show that these auxiliary semigroups have the integral representations
\begin{gather*}
\mathbb{T}_t^{\a,\eta,+} f (x) =\int_{\RRR} \mathbb{G}_t^{\a,\eta,+} (x,y) f(y) \, d\mu_\a^+(y), \qquad x\in \RRR, \qquad t>0,\\
\mathbb{G}_t^{\a,\eta,+} (x,y) = \exp\big(-2|\eta|t\big) (xy)^\eta G_t^{\a + \eta} (x,y), \qquad x,y \in \RRR, \qquad t>0.
\end{gather*}
Further, these series/integral formulas provide a good def\/inition of $\mathbb{T}_t^{\a,\eta,+}$ on weighted $L^p$ spaces for a large class of weights and produce always smooth functions of $(x,t) \in \RRR \times \mathbb{R}_+$, see Proposition~\ref{pro:conv}. Again, note that choosing $\eta_0=\boldsymbol{0}$ we have $\mathbb{T}_t^{\a,\eta_0,+} =T^{\a}_t$.

For $\eta \in \{ 0,1 \}^d$ and $n \in \N^d$ we denote
\begin{gather*}
\dersym_{\eta,n}=\dersym_{d,\eta_d, n_d} \circ \cdots \circ \dersym_{1,\eta_1, n_1},\qquad \dersym_{i,\eta_i, n_i}
= \dersym_{i, \overline{\eta_i + n_i - 1} }\circ \cdots \circ \dersym_{i, \overline{\eta_i + 1} } \circ\dersym_{i, \overline{\eta_i} } ,
\end{gather*}
(by convention, $\dersym_{i,\eta_i, 0} = \Id$), where for $i \in \{1, \ldots, d \}$ and $\eta_i \in \{0, 1\}$ we put, see \eqref{pp},
\begin{gather*}
\dersym_{i, \eta_i } = \partial_{i, \eta_i } + (-1)^{\eta_i}x_i.
\end{gather*}
Notice that the derivative $\dersym_{\eta, n}$ corresponds to the action of $\mathbb{D}^{n}$ on $\eta$-symmetric functions. Precisely, if $f$ is $\eta$-symmetric, then $\mathbb{D}^{n} f = \dersym_{\eta, n} f$. Further, we may also think that $\dersym_{\eta, n}$ acts on functions def\/ined on the smaller space $\RRR$.

Now we are ready to introduce the auxiliary Laguerre-type operators, which are def\/ined initially in $L^2(d\mu_\a^+)$. For each $\eta \in \{0,1\}^d$ we def\/ine the following objects.
\begin{description}\itemsep=0pt
\item[{\rm (L-t$'$.I)}] The Laguerre-type heat semigroup maximal operator
\begin{gather*}
\mathbb{T}_{*}^{\alpha,\eta, +} f =\big\| \mathbb{T}_t^{\alpha,\eta, +} f\big\|_{L^{\infty}(\R_+,dt)}.
\end{gather*}
\item[{\rm (L-t$'$.II)}] Laguerre-type Riesz transforms of order $|n|>0$
\begin{gather*}
\mathbb{R}_{n}^{\alpha,\eta, +} f = \sum_{ k \in \N^d, \, \overline{k} = \eta}
\big( \lambda_{|\lfloor (k+\boldsymbol{1})/2 \rfloor |}^{\a} \big)^{-|n|\slash 2}
	\langle f^\eta , \Phi_k^{\a} \rangle_{d\mu_{\alpha}}
\big( \dersym_{\eta, n} \Phi_k^{\a} \big)^+,
\end{gather*}
where $n \in \N^d \setminus \{\boldsymbol{0}\}$. Observe that if $\overline{k} = \eta$, then $\dersym_{\eta, n} \Phi_k^{\a} = \mathbb{D}^{n} h_k^{\a}$ is $(\overline{k + n})$-symmetric and hence the quantity $\big( \dersym_{\eta, n} \Phi_k^{\a} \big)^+$ is well def\/ined.
\item[{\rm (L-t$'$.III)}] Multipliers of Laplace and Laplace--Stieltjes transform types
\begin{gather*}
\mathbb{M}^{\alpha, \eta, +}_{\mathfrak{m}} f = \sum_{ k \in \N^d, \, \overline{k} = \eta}
\mathfrak{m}\big( \lambda_{|\lfloor (k+\boldsymbol{1})/2 \rfloor |}^{\a} \big)	\langle f^\eta , \Phi_k^{\a} \rangle_{d\mu_{\alpha}}
\big( \Phi_k^{\a} \big)^+,
\end{gather*}
where $\mathfrak{m}$ is as in (L-D.III).
\item[{\rm (L-t$'$.IV)}] Littlewood--Paley--Stein type mixed $g$-functions
\begin{gather*}
\mathbb{g}^{\alpha,\eta,+}_{n,m}(f) = \big\| \partial_t^m \dersym_{\eta, n} \mathbb{T}_t^{\alpha,\eta, +}f \big\|_{L^2(\R_+,t^{|n|+2m-1}dt)},
\end{gather*}
where $n \in \N^d$, $m \in \N$ are such that $|n|+m>0$.
\item[{\rm (L-t$'$.V)}] Mixed Lusin area integrals
\begin{gather*}
\mathbb{S}^{\a,\eta,+}_{n,m}(f)(x)= \left( \int_{A(x)\cap \R^{d+1}_+} t^{|n| + 2m -1} \big|\partial_{t}^m
\dersym_{\eta, n} \mathbb{T}_t^{\alpha,\eta, +} f (z) \big|^{2}\frac{d\mu_{\a}^+(z) \, dt}
{V_{\sqrt{t}}^{\a,+}(x)} \right)^{1\slash 2},
\end{gather*}
where $n \in \N^d$, $m \in \N$ are such that $|n|+m>0$, and $A(x)$, $V_{t}^{\a,+}(x)$ are def\/ined in~\eqref{def:A} and~\eqref{def:V+}, respectively.
\end{description}

Notice that the Laguerre-type Lusin area integrals can be written as
\begin{gather*}
\mathbb{S}^{\a,\eta,+}_{n,m}(f)(x)=\big\| \partial_t^m
\dersym_{\eta, n} \mathbb{T}_t^{\alpha,\eta, +} f (x+z) \sqrt{\Xi_{\a}(x,z,t)}
 \chi_{\{x+z \in \RRR \}}\big\|_{L^2(A,t^{ |n| + 2m - 1} dzdt)},
\end{gather*}
where $\Xi_{\a}$ is def\/ined in \eqref{def:Xi}. We note that the series def\/ining $\mathbb{R}_{n}^{\alpha,\eta,+}$ and $\mathbb{M}^{\alpha,\eta,+}_{\mathfrak{m}}$ converge in $L^2(d\mu_\a^+)$ and produce $L^2(d\mu_\a^+)$-bounded operators. This follows from the $L^2(d\mu_\a)$-boundedness of $\mathbb{R}_{n}^{\alpha}$ and $\mathbb{M}^{\alpha}_{\mathfrak{m}}$, for symmetry reasons. Further, the formulas def\/ining $\mathbb{T}^{\alpha,\eta,+}_{*}f$, $\mathbb{g}^{\alpha,\eta,+}_{n,m}(f)$ and $\mathbb{S}^{\a,\eta,+}_{n,m}(f)$ make a pointwise sense for general functions~$f$, see Proposition~\ref{pro:conv}.

Arguments similar to those given in \cite[p.~6]{NoSt4} and \cite[pp.~1522--1524]{TZS3} allow us to reduce the proof of Theorem~\ref{Sthm:main} to showing the following.

\begin{Theorem}\label{Sthm:main+}
Assume that $\a \in(-1,\infty)^d$ and $\eta \in \{ 0, 1\}^d$. Then the Laguerre-type operators {\rm (L-t$'$.II)}, {\rm (L-t$'$.III)} extend to bounded linear operators on $L^{p}(Ud\mu^+_{\a})$, $U \in A_{p}^{\a,+}$, $1<p<\infty$, and from $L^{1}(Ud\mu^+_{\a})$ to weak $L^{1}(Ud\mu^+_{\a})$, $U \in A_{1}^{\a,+}$.
Furthermore, the sublinear operators {\rm (L-t$'$.I)}, {\rm (L-t$'$.IV)} and {\rm (L-t$'$.V)} are bounded on $L^{p}(Ud\mu^+_{\a})$, $U \in A_{p}^{\a,+}$, $1<p<\infty$, and from $L^1(Ud\mu^+_{\a})$ to weak $L^1(Ud\mu^+_{\a})$, $U \in A_{1}^{\a,+}$.
\end{Theorem}

To prove Theorem~\ref{Sthm:main+} we will use the general Calder\'on--Zygmund theory. We will just show that the Laguerre-type operators (L-t$'$.I)--(L-t$'$.V) are (vector-valued) Calder\'on--Zygmund operators in the sense of the space of homogeneous type $(\RRR,\mu_\a^+,\|\cdot\|)$, see the comment following the statement of Theorem~\ref{thm:main+}. The result below implies Theorem~\ref{Sthm:main+} and thus also Theorem~\ref{Sthm:main}.

\begin{Theorem}\label{Sthm:CZ}
Assume that $\a \in (-1,\infty)^{d}$ and $\eta\in \{ 0 , 1 \}^d$. The Laguerre-type Riesz transforms~{\rm (L-t$'$.II)} and the multipliers of Laplace and Laplace--Stieltjes transform types {\rm (L-t$'$.III)} are scalar-valued Calder\'on--Zygmund operators in the sense of the space $(\RRR,\mu_{\a}^+,\|\cdot\|)$.
Furthermore, the Laguerre-type heat semigroup maximal operator \emph{(L-t'.I)}, the mixed $g$-functions \mbox{{\rm (L-t$'$.IV)}} and the mixed Lusin area integrals {\rm (L-t$'$.V)} can be viewed as vector-valued Calder\'on--Zygmund operators in the sense of $(\RRR,\mu_{\a}^+,\|\cdot\|)$ associated with the Banach spaces $\mathbb{B}=C_0$, $\mathbb{B}=L^2(\R_+,t^{|n|+2m-1}dt)$ and $\mathbb{B}=L^2(A,t^{|n|+2m-1}dzdt)$, respectively.
\end{Theorem}

Proving Theorem~\ref{Sthm:CZ} splits into showing the following three results.

\begin{Proposition}\label{Spro:L2b}
Let $\a\in(-1,\infty)^d$ and $\eta\in \{ 0 , 1 \}^d$. Then the Laguerre-type operators from Theorem~{\rm \ref{Sthm:CZ}} are bounded on $L^2(d\mu_{\a}^{+})$.
\end{Proposition}

\begin{proof}
The $L^2(d\mu_{\a}^+)$-boundedness of $\mathbb{R}_n^{\a,\eta,+}$ and $\mathbb{M}_{\mathfrak{m}}^{\a,\eta,+}$ is already justif\/ied, see the comments preceding Theorem~\ref{Sthm:main+}. Since $\mathbb{T}_*^{\a,\eta,+}f$ is controlled pointwise by $\mathfrak{T}_*^{\a,\eta,+} |f|$ (this is immediately seen by comparing the associated integral kernels), the $L^2(d\mu_{\a}^+)$-boundedness of $\mathbb{T}_*^{\a,\eta,+}$ follows from Proposition~\ref{pro:L2b}. The square functions are dealt with similarly as their Laguerre--Dunkl counterparts in the proof of Proposition~\ref{pro:L2b}. More precisely, the case of $\mathbb{S}_{n,m}^{\a,\eta,+}$ is reduced to~$\mathbb{g}_{n,m}^{\a,\eta,+}$ with the aid of Lemma~\ref{lem:intLusin}(a). On the other hand, the $L^2(d\mu_{\a}^+)$-boundedness of $\mathbb{g}_{n,m}^{\a,\eta,+}$ is a consequence of the $L^2(d\mu_{\a})$-boundedness of $\mathbb{g}_{n,m}^{\a}$ (via restricting to $\eta$-symmetric functions). The latter property is verif\/ied in an analogous way to the same for $\mathfrak{g}_{n,m,\om}^{\a}$, see the proof of Proposition~\ref{pro:L2b}, where now one should use~\eqref{C} and~\eqref{D} instead of~\eqref{A} and~\eqref{B}.
\end{proof}

Formal computations suggest that the Laguerre-type operators in question are associated with the following kernels related to appropriate Banach spaces $\mathbb{B}$.
\begin{description}\itemsep=0pt
\item[{\rm (L-t$'$.I)}] The kernel associated with the Laguerre-type heat semigroup maximal operator,
\begin{gather*}
\mathbb{U}^{\alpha, \eta, +}(x,y) = \big\{\mathbb{G}_t^{\alpha,\eta,+}(x,y)\big\}_{t>0}, \qquad \mathbb{B}= C_0 \subset L^{\infty}(\R_+,dt).
\end{gather*}
\item[{\rm (L-t$'$.II)}] The kernels associated with the Laguerre-type Riesz transforms,
\begin{gather*}
\mathbb{R}_{n}^{\alpha,\eta,+}(x,y) = \frac{1}{\Gamma(|n|\slash 2)} \int_0^{\infty} \dersym_{\eta, n, x}
\mathbb{G}_t^{\alpha, \eta, +}(x,y)
	t^{|n|\slash 2 -1}\, dt, \qquad \mathbb{B}=\mathbb{C},
\end{gather*}
where $n \in \N^d \setminus \{\boldsymbol{0}\}$.
\item[{\rm (L-t$'$.IIIa)}] The kernels associated with the Laplace transform type multipliers,
\begin{gather*}
\mathbb{K}^{\alpha, \eta, +}_{\psi}(x,y)= - \int_0^{\infty} \psi(t) \partial_t \mathbb{G}_t^{\alpha, \eta, +}(x,y)\, dt, \qquad \mathbb{B}=\mathbb{C},
\end{gather*}
where $\psi \in L^{\infty}(\R_+,dt)$.
\item[{\rm (L-t$'$.IIIb)}] The kernels associated with the Laplace--Stieltjes transform type multipliers,
\begin{gather*}
\mathbb{K}^{\alpha, \eta, +}_{\nu}(x,y) = \int_{\R_+} \mathbb{G}_t^{\alpha, \eta, +}(x,y)\, d\nu(t), \qquad \mathbb{B}=\mathbb{C},
\end{gather*}
where $\nu$ is a signed or complex Borel measure on $\R_+$ with the total variation $|\nu|$ satis\-fying~\eqref{assum}.
\item[{\rm (L-t$'$.IV)}] The kernels associated with the mixed $g$-functions,
\begin{gather*}
\mathbb{H}^{\alpha, \eta, +}_{n,m}(x,y) = \big\{ \partial_t^m \dersym_{\eta, n, x}
	\mathbb{G}_t^{\alpha, \eta, +}(x,y) \big\}_{t>0}, \qquad \mathbb{B} = L^2(\R_+,t^{|n|+2m-1}dt),
\end{gather*}
where $n \in \N^d$, $m \in \N$ are such that $|n|+m>0$.
\item[{\rm (L-t$'$.V)}] The kernels associated with the mixed Lusin area integrals
\begin{gather*}
\mathbb{S}^{\alpha, \eta, +}_{n,m}(x,y) =
\left\{
\partial_t^m \dersym_{\eta, n, \mathbf{x}} \mathbb{G}_{t}^{\a,\eta,+}(\mathbf{x},y)
\Big|_{\mathbf{x} = x+z} \sqrt{\Xi_{\a}(x,z,t)}
\chi_{\{x+z\in\RRR\}} \right\}_{(z,t) \in A}
\end{gather*}
with $\mathbb{B} = L^2(A,t^{|n|+2m-1}dzdt)$, where $n \in \N^d$, $m \in \N$ are such that $|n|+m>0$.
\end{description}

\begin{Proposition}\label{Spro:kerassoc}
Assume that $\a \in (-1, \infty)^d$ and $\eta \in \{ 0 , 1 \}^d$. Then the Laguerre-type operators {\rm (L-t$'$.I)--(L-t$'$.V)} are associated, in the Calder\'on--Zygmund theory sense, with the corresponding kernels listed above.
\end{Proposition}

\begin{proof}
Here the reasoning is similar to that in the Laguerre--Dunkl setting, see the proof of Proposition~\ref{pro:kerassoc} where the case of $g$-functions is explained. We leave details to the reader.
\end{proof}

\begin{Theorem}\label{Sthm:kerest}
Let $\a \in (-1, \infty)^d$ and $\eta \in \{ 0 , 1 \}^d$. Then the kernels {\rm (L-t$'$.I)--(L-t$'$.V)} listed above satisfy the standard estimates
with the relevant Banach spaces $\mathbb{B}$. More precisely, the kernels {\rm (L-t$'$.I)--(L-t$'$.IV)} satisfy the smoothness conditions with $\gamma = 1$, and the kernel {\rm (L-t$'$.V)} satisfies~\eqref{sm1} and~\eqref{sm2} with any $\gamma \in (0,1/2]$ such that $\gamma < \min\limits_{1 \le i \le d} (\a_i + 1)$.
\end{Theorem}

For the proof of Theorem~\ref{Sthm:kerest} see the end of Section~\ref{sec:ker}.

\subsection{Laguerre setting} \label{ssec:lag}

In this subsection we state new results in the Laguerre setting that are mostly implied by our Laguerre--Dunkl and Laguerre-symmetrized results. Recall that $\delta_i$, $i=1,\ldots,d$ (but not $\delta_i^*$), are the appropriate f\/irst order derivatives in the Laguerre context. Thus, at f\/irst glance, a natural choice of higher-order derivatives is simply
\begin{gather*}
\delta^n = \delta_d^{n_d} \circ \cdots \circ \delta_1^{n_1}, \qquad n \in \N^d.
\end{gather*}
Higher-order Riesz--Laguerre transforms and mixed $g$-functions involving $\delta^n$ were investigated by the authors in \cite{NoSt2,NoSz}. However, as it was pointed out in~\cite{NoSt5} by the f\/irst and second authors, seemingly even more natural higher-order derivatives in this situation are the
interlaced derivatives
\begin{gather*}
D^n = \big(\underbrace{\cdots \delta_d \delta_d^* \delta_d \delta_d^* \delta_d}_{n_d\;
	\textrm{components}} \big)
	\circ
	\cdots
	\circ
	\big(\underbrace{\cdots \delta_2 \delta_2^* \delta_2 \delta_2^* \delta_2}_{n_2\; \textrm{components}} \big)
	\circ
	\big(\underbrace{\cdots \delta_1 \delta_1^* \delta_1 \delta_1^* \delta_1}_{n_1\; \textrm{components}} \big).
\end{gather*}
Therefore we now consider the following operators def\/ined via $D^n$ and given initially in $L^2(d\mu_{\a}^+)$.
\begin{description}\itemsep=0pt
\item[{\rm (L.II)}] Riesz--Laguerre transforms of order $|n|>0$
\begin{gather*}
{R}_{n}^{\alpha} f = \sum_{ k \in \N^d}
\big( \lambda_{|k|}^{\a} \big)^{-|n|\slash 2}
	\langle f , \ell_k^{\a} \rangle_{d\mu^+_{\alpha}}\,
	D^n\ell_k^{\a},
\end{gather*}
where $n \in \N^d \setminus \{\boldsymbol{0}\}$.
\item[{\rm (L.IV)}] Littlewood--Paley--Stein type mixed $g$-functions
\begin{gather*}
{g}^{\alpha}_{n,m}(f) = \big\| \partial_t^m D^n {T}_t^{\alpha}f \big\|_{L^2(\R_+,t^{|n|+2m-1}dt)},
\end{gather*}
where $n \in \N^d$, $m \in \N$ are such that $|n|+m>0$.
\item[{\rm (L.V)}] Mixed Lusin area integrals
\begin{gather*}
{S}^{\a}_{n,m}(f)(x)=
\left( \int_{A(x)\cap \R^{d+1}_+} t^{|n| + 2m -1} \big|\partial_{t}^m
D^n {T}_t^{\alpha} f(z) \big|^{2} \frac{d\mu_{\a}^+(z) \, dt}
{V_{\sqrt{t}}^{\a,+}(x)} \right)^{1\slash 2},
\end{gather*}
where $n \in \N^d$, $m \in \N$ are such that $|n|+m>0$, and $A(x)$, $V_t^{\a,+}(x)$ are def\/ined in~\eqref{def:A} and~\eqref{def:V+}, respectively.
\end{description}

Observe that the operators (L.II), (L.IV) and (L.V) coincide (up to the sign $(-1)^{|\lfloor n/2 \rfloor|}$ in case of (L.II)) with the Laguerre-type operators \mbox{(L-t$'$.II)}, (L-t$'$.IV) and (L-t$'$.V) with $\eta=\boldsymbol{0}$ investigated in Section~\ref{ssec:L-s}. Furthermore, they also coincide with the Laguerre-type operators \mbox{(L-t.II)}, (L-t.IV) and (L-t.V) with $\eta=\boldsymbol{0}$ and $\om = \om_0 \in \N^{|n|}$ such that $\om_j^i = (-1)^{j+1}$, $i \in \{1,\ldots,d\}$, $j \in \{1,\ldots,n_i\}$ studied in Section~\ref{ssec:L-D}. Thus we know that the series def\/ining $R^{\a}_n$ converges in $L^2(d\mu_{\a}^+)$, and the formulas def\/ining $g_{n,m}^{\a}(f)$ and $S_{n,m}^{\a}(f)$ can be understood pointwise for $f \in L^p(Ud\mu_{\a}^{+})$, $U \in A_p^{\a,+}$, $1\le p < \infty$. Moreover, the following result holds (see Theorems~\ref{Sthm:main+} and~\ref{thm:main+}).

\begin{Theorem}\label{thm:lag}
Assume that $\a \in(-1,\infty)^d$. Then the Riesz--Laguerre transforms {\rm (L.II)} extend to bounded linear operators on $L^{p}(Ud\mu^+_{\a})$, $U \in A_{p}^{\a,+}$, $1<p<\infty$, and from $L^{1}(Ud\mu^+_{\a})$ to weak $L^{1}(Ud\mu^+_{\a})$, $U \in A_{1}^{\a,+}$. Furthermore, the square functions {\rm (L.IV)} and {\rm (L.V)} are bounded sublinear operators on $L^{p}(Ud\mu^+_{\a})$, $U \in A_{p}^{\a,+}$, $1<p<\infty$, and from $L^1(Ud\mu^+_{\a})$ to weak
$L^1(Ud\mu^+_{\a})$, $U \in A_{1}^{\a,+}$.
\end{Theorem}

Concerning the Riesz--Laguerre transforms, Theorem~\ref{thm:lag} complements the analogous result for the above mentioned other variant of higher-order Riesz--Laguerre transforms obtained in \cite[Corollary~4.2]{NoSz} and earlier in \cite[Theorem~3.8]{NoSt2} under the restriction $\a \in [-1/2,\infty)^d$.

The Laguerre $g$-functions $g_{n,m}^{\a}$ and the Lusin area integrals $S_{n,m}^{\a}$ of order $1$ (i.e., in the cases when $|n|+m=1$) were studied earlier f\/irst by the third-named author under the restriction $\a \in [-1/2,\infty)^d$; see~\cite[Corollary~2.5]{TZS1} and~\cite[Theorem~2.8]{TZS3}, which
cover the just indicated special cases of Theorem~\ref{thm:lag}. An analogue of Theorem~\ref{thm:lag} for the variant of higher-order Laguerre $g$-functions def\/ined via $\delta^n$ rather than $D^n$ is contained in \cite[Corollary~4.2]{NoSz}.

We point out that the results of Sections~\ref{ssec:L-D} and~\ref{ssec:L-s} readily imply further generalizations of known results in the Laguerre setting. For instance, choosing $\eta=e_i$, $i=1,\ldots,d$, in Theorem~\ref{Sthm:main+} one recovers and generalizes the results pertaining to Riesz transforms and square functions related to the so-called modif\/ied Laguerre semigroups found in \cite[p.~664]{NoSt2}, \cite[Corollary~2.5]{TZS1} and \cite[Theorem~2.8]{TZS3}. We leave further details to interested readers.

It is worth mentioning that Theorem~\ref{thm:lag} can be seen as a direct consequence of either Theo\-rem~\ref{thm:main} or Theorem~\ref{Sthm:main}. Indeed, it suf\/f\/ices to restrict the operators in the latter two theorems to ref\/lection invariant functions and choose $\om=\om_0$ appearing implicitly in Theorem~\ref{thm:main}.

Finally, we take this opportunity to complement the results of \cite{NoSz, TZS1,TZS3} by providing an analogue of Theorem~\ref{thm:lag} for the Laguerre mixed Lusin area integrals def\/ined via $\delta^n$ rather than~$D^n$. Thus we consider
\begin{description}\itemsep=0pt
\item[{\rm (L.VI)}] Mixed Lusin area integrals
\begin{gather*}
{s}^{\a}_{n,m}(f)(x)=\left( \int_{A(x)\cap \R^{d+1}_+} t^{|n| + 2m -1} \big|\partial_{t}^m
\delta^n {T}_t^{\alpha} f(z) \big|^{2} \frac{d\mu_{\a}^+(z) \, dt} {V_{\sqrt{t}}^{\a,+}(x)} \right)^{1\slash 2},
\end{gather*}
where $n \in \N^d$, $m \in \N$ are such that $|n|+m>0$, and $A(x)$, $V_t^{\a,+}(x)$ are def\/ined in~\eqref{def:A} and~\eqref{def:V+}, respectively.
\end{description}
Since $T_t^{\a}f(z)$ is a smooth function of $(z,t)\in \R^d_+\times\R_+$ whenever $f\in L^p(Ud\mu_{\a}^+)$, $U \in A_p^{\a,+}$, $1 \le p < \infty$ (see \cite[p.~811]{NoSz}), this def\/inition makes pointwise sense for the general~$f$ as above.
\begin{Theorem}\label{thm:lags}
Assume that $\a \in(-1,\infty)^d$. The Lusin area integrals {\rm (L.VI)} are bounded sublinear operators on $L^{p}(Ud\mu^+_{\a})$, $U \in A_{p}^{\a,+}$, $1<p<\infty$, and from $L^1(Ud\mu^+_{\a})$ to weak $L^1(Ud\mu^+_{\a})$, $U \in A_{1}^{\a,+}$.
\end{Theorem}

The proof of Theorem~\ref{thm:lags} goes along yet familiar lines of the Calder\'on--Zygmund theory. One views $s_{n,m}^{\a}$ as vector-valued linear operators associated with
\begin{description}\itemsep=0pt
\item[{\rm (L.VI)}] The kernels associated with mixed Lusin area integrals
\begin{gather*}
{s}^{\alpha}_{n,m}(x,y)=\left\{
\partial_t^m \delta^n_{\mathbf{x}} {G}_{t}^{\a}(\mathbf{x},y)
\Big|_{\mathbf{x} = x+z} \sqrt{\Xi_{\a}(x,z,t)}\chi_{\{x+z\in\RRR\}} \right\}_{(z,t) \in A}
\end{gather*}
\end{description}
taking values in the Banach space $\mathbb{B} = L^2(A,t^{|n|+2m-1}dzdt)$. Such operators are bounded from $L^2(d\mu_{\a}^+)$ to $L^2_{\mathbb{B}}(d\mu_{\a}^+)$, which by means of Lemma~\ref{lem:intLusin}(a) below is a consequence of the $L^2(d\mu_{\a}^+)$-boundedness of mixed $g$-functions def\/ined via $\delta^n$; see the proof of \cite[Theorem~4.1]{NoSz}. The fact that $s_{n,m}^{\a}$ is indeed associated with the kernel
$s_{n,m}^{\a}(x,y)$ is verif\/ied in a similar manner as for the Laguerre-type Lusin area integrals considered in Sections~\ref{ssec:L-D} and~\ref{ssec:L-s}. Finally, the standard estimates for $s_{n,m}^{\a}(x,y)$ follow from estimates obtained in~\cite{NoSz} by means of the strategy established in Section~\ref{sec:ker}. More precisely, observe that with the aid of \cite[Lemma~2.4]{NoSz} for $l,r \in \N^d$ such that $|l| + |r| \le 1$ the quantity
$\big| \partial_y^r \partial_x^l \partial^m_t \delta^n_x G_t^\a (x,y) \big|$ is controlled by the right-hand side of the bound appearing in the statement of Lemma~\ref{lem:heatEST} with $\eta = \boldsymbol{0}$ and the exponent of $\e\bb$ replaced by $1/4$ there (the exact value of this exponent is meaningless for our developments). Then, proceeding as in the proof of Theorem~\ref{thm:kerest} (the case of $\mathfrak{S}^{\alpha, \eta, +}_{n,m,\om}(x,y)$) we obtain the standard estimates for $s_{n,m}^{\a}(x,y)$; see also the proof of Theorem~\ref{Sthm:kerest} in Section~\ref{sec:ker}.

\subsection{Further results and comments} \label{ssec:fur}

For the sake of brevity, we shall focus here on the Laguerre--Dunkl setting. Nevertheless, eve\-ry\-thing what follows in this subsection, in particular the forthcoming theorem, after suitable and quite obvious modif\/ications pertains also to the Laguerre-symmetrized and the Laguerre contexts.

The Laguerre--Dunkl Poisson semigroup $\mathfrak{P}_t^{\a} = \exp(-t\sqrt{\mathfrak{L}_{\a}})$, $t>0$, is related to the Laguerre--Dunkl heat semigroup via the subordination formula
\begin{gather} \label{subord}
\mathfrak{P}_t^{\a}f(x) = \int_0^{\infty} \mathfrak{T}^{\a}_{t^2/(4u)}f(x)\frac{e^{-u}\, du}{\sqrt{\pi u}}
\end{gather}
valid pointwise for $f \in L^p(Wd\mu_{\a})$, $W \in A_p^{\a}$, $1\le p < \infty$. It is of interest and importance to investigate counterparts of the operators (L-D.I) and (L-D.III)--(L-D.V) associated with the Poisson semigroup. More precisely, these are the following.
\begin{description}\itemsep=0pt
\item[{\rm (L-D.P.I)}] The Laguerre--Dunkl Poisson semigroup maximal operator
\begin{gather*}
\mathfrak{P}_{*}^{\alpha}f = \big\| \mathfrak{P}_t^{\alpha}f\big\|_{L^{\infty}(\R_+,dt)}.
\end{gather*}
\item[{\rm (L-D.P.III)}] Multipliers of Laplace and Laplace--Stieltjes transform types
\begin{gather*}
\mathfrak{M}^{\alpha,P}_{\mathfrak{m}} f = \sum_{k \in \N^d}
\mathfrak{m}\Big( \sqrt{\lambda_{|k|/2}^{\a}}\, \Big)
		\langle f , h_k^{\a} \rangle_{d\mu_{\alpha}}\, h_k^{\a},
\end{gather*}
where either $\mathfrak{m}(z) = z\int_0^{\infty} e^{-tz} \psi(t)\, dt$ with $\psi \in L^{\infty}(\R_+,dt)$ or $\mathfrak{m}(z) = \int_{\R_+} e^{-tz} \, d\nu (t)$ with~$\nu$ being a signed or complex Borel measure on~$\R_+$, with its total variation $|\nu|$ satisfying
\begin{gather*}
\int_{\R_+} e^{-t \sqrt{\lambda_{0}^{\a}} } \, d|\nu|(t) < \infty.
\end{gather*}
\item[{\rm (L-D.P.IV)}] Littlewood--Paley--Stein type mixed $g$-functions
\begin{gather*}
\mathfrak{g}^{\alpha,P}_{n,m,\om}(f) = \big\| \partial_t^m \mathfrak{D}^{n,\om}
\mathfrak{P}_t^{\alpha}f \big\|_{L^2(\R_+,t^{2|n|+2m-1}dt)},
\end{gather*}
where $n \in \N^d$, $m \in \N$ are such that $|n|+m>0$, and $\om \in \{-1,1\}^{|n|}$.
\item[{\rm (L-D.P.V)}] Mixed Lusin area integrals
\begin{gather*}
\mathfrak{S}^{\a,P}_{n,m,\om}(f)(x)=
\left( \int_{C(x)} t^{2|n| + 2m -1} \big|\partial_{t}^m \mathfrak{D}^{n,\om}
\mathfrak{P}_t^{\alpha}f (z)\big|^{2}\frac{d\mu_{\a}(z) \, dt}{V_{{t}}^{\a}(x)} \right)^{1\slash 2},
\end{gather*}
where $n \in \N^d$, $m \in \N$ are such that $|n|+m>0$, and $\om \in \{-1,1\}^{|n|}$. Further, $V_t^{\a}$ is as in~\eqref{def:V}, and $C(x)$ is the standard cone with vertex at $x$,
\begin{gather*}
C(x)=(x,0)+C,\qquad C=\big\{(z,t)\in \RR \times\R_+ \colon |z|<{t}\big\}.
\end{gather*}
\end{description}

The exact aperture of $C$ is of course irrelevant for our considerations. It is easily seen that the series def\/ining $\mathfrak{M}^{\alpha,P}_{\mathfrak{m}}$ converges in $L^2(d\mu_{\a})$ and produces $L^2(d\mu_\a)$-bounded operator. The remaining operators are well def\/ined pointwise for $f \in L^p(Wd\mu_{\a})$, $W\in A_p^{\a}$, $1 \le p < \infty$.

The techniques presented in this paper combined with the subordination formula \eqref{subord} allow one to show the following result.
\begin{Theorem}\label{thm:mainP}
Assume that $\a \in(-1,\infty)^d$ and $W$ is a weight on $\RR$ invariant under the reflections $\sigma_1, \ldots ,\sigma_d$. Then the multipliers of Laplace and Laplace--Stieltjes transform types \mbox{{\rm (L-D.P.III)}} extend to bounded linear operators on $L^{p}(Wd\mu_\a)$, $W\in A_{p}^{\a}$, $1<p<\infty$, and from $L^1 (Wd\mu_\a)$ to weak $L^1 (Wd\mu_\a)$, $W\in A_{1}^{\a}$. Furthermore, the Laguerre--Dunkl Poisson semigroup maximal operator {\rm (L-D.P.I)}, the mixed $g$-functions {\rm (L-D.P.IV)} and the mixed Lusin area integrals {\rm (L-D.P.V)} are bounded on $L^{p}(Wd\mu_\a)$, $W\in A_{p}^{\a}$, $1<p<\infty$, and from $L^1 (Wd\mu_\a)$ to weak $L^1 (Wd\mu_\a)$, $W\in A_{1}^{\a}$.
\end{Theorem}

Actually, the part of Theorem~\ref{thm:mainP} concerning the maximal operator and $g$-functions is a~direct consequence of the subordination formula and Theorem~\ref{thm:main}. To prove the remaining part one has to combine the Calder\'on--Zygmund operator theory approach with~\eqref{subord}. We leave the details to interested readers. For more hints on the way of reasoning in case of the Lusin area integrals we refer to the proof of \cite[Proposition~2.8]{CaSz}. Here we just note that instead of using Lemma~\ref{lem:qz} (see \cite[p.~963, line~7]{CaSz}) it is convenient to use the fact that there exist constants $c_1,c_2 > 0$ such that the estimate
\begin{gather*}
\left( 1 - \z^2\left( \frac{t}{4u} \right) \right)^{(d+|\a|)/2} \left( \e \left(\z \left( \frac{t}{4u} \right) ,q_{\pm}(x+z,y,s) \right) \right)^{1/2} \\
\qquad {}\lesssim
\left( \e \left(\z \left( \frac{t}{4u} \right) ,q_{\pm}(x,y,s) \right) \right)^{c_1} e^{u/2}
+ \chi_{ \{ t \ge 4u \} } \exp \left( {-}c_2 \frac{|x|^2 + |y|^2}{u+1} \right)
\end{gather*}
holds uniformly in $u > 0$, $x,y \in \RRR$, $(z,t) \in A$ and $s \in (-1,1)^d$.

{\sloppy Theorem~\ref{thm:mainP} extends, in particular, \cite[Corollary~2.8]{TZS2} where the Laplace multipliers \mbox{(L-D.P.III)} were treated under the assumption $\a \in [-1/2,\infty)^d$. Further, it generalizes \cite[Theorem~2.7]{TZS3} dealing with the $g$-functions (L-D.P.IV) of order $1$, i.e., with $|n|+m=1$, and under the restriction $\a \in [-1/2,\infty)^d$. The Hermite case $\a=-\boldsymbol{1}/2$ was treated earlier in \cite{HRST,StTo2} (but only order $1$, no weights and $1<p<\infty$ in case of \cite{HRST}). The Lusin area integrals (L-D.P.V) were not studied earlier, except for the one-dimensional Hermite case investigated in~\cite{BMR} (only order~$1$ with $n=0$ and $m=1$).

}

Next, we comment on further operators to which the methods of this paper apply. Here we focus only on square functions, since more general variants of these seem to be of greatest importance. Consider then the following operators.
\begin{description}\itemsep=0pt
\item[{\rm (L-D.IV.gen)}] Fractional Littlewood--Paley--Stein type mixed $g$-functions
\begin{gather*}
\mathfrak{g}^{\alpha}_{n,\gamma,\om,r}(f) = \big\| \partial_t^{\gamma} \mathfrak{D}^{n,\om}
\mathfrak{T}_t^{\alpha}f \big\|_{L^r(\R_+,t^{(|n|/2+\gamma)r-1}dt)},
\end{gather*}
where $n \in \N^d$, $\gamma \ge 0$, $|n|+\gamma>0$, $\om \in \{-1,1\}^{|n|}$ and $2 \le r < \infty$.
\item[{\rm (L-D.V.gen)}] Fractional mixed Lusin area integrals
\begin{gather*}
\mathfrak{S}^{\a}_{n,\gamma,\om,r}(f)(x)= \left( \int_{A(x)} t^{(|n|/2 + \gamma)r -1} \big|\partial_{t}^{\gamma} \mathfrak{D}^{n,\om}
\mathfrak{T}_t^{\alpha}f (z)\big|^{r}\frac{d\mu_{\a}(z) \, dt} {V_{\sqrt{t}}^{\a}(x)} \right)^{1\slash r},
\end{gather*}
where $n \in \N^d$, $\gamma \ge 0$, $|n|+\gamma>0$, $\om \in \{-1,1\}^{|n|}$, $2 \le r < \infty$, and $A(x)$ and $V_{t}^{\a}(x)$ are as in~\eqref{def:A} and~\eqref{def:V}, respectively.
\end{description}
Here $\gamma$ may not be integer, and $\partial_t^{\gamma}$ denotes the Caputo type fractional derivative given by
\begin{gather*}
\partial_t^{\gamma}F(t) = \frac{1}{\Gamma(m-\gamma)} \int_0^{\infty} \partial_t^{m}F(t+s)s^{m-\gamma-1}\,ds, \qquad t > 0,
\end{gather*}
for suitable $F$, with integer $m=\lfloor \gamma \rfloor + 1$. Natural counterparts of (L-D.IV.gen) and (L-D.V.gen) based on the Laguerre--Dunkl Poisson semigroup also come into play. For all these operators a~result analogous to Theorem~\ref{thm:main} can be shown, though it requires further analysis that is beyond the scope of this paper.

The study of $g$-functions involving fractional derivatives goes back to Segovia and Wheeden~\cite{SW}. More recently square functions def\/ined via fractional derivatives were investigated in~\cite{Bpot1,Bpot2,Lpot} in connection with potential spaces in various settings. On the other hand, extensions relying on taking any $r \ge 2$ rather than the standard $r=2$ are quite natural and well known in the literature; see, for instance,~\cite{BCN,BMR}.

Finally, we point out that the general Calder\'on--Zygmund theory covers also more subtle mapping properties comparing to $L^p$-boundedness and weak type $(1,1)$ entering the main results of this paper. This remark concerns, in particular, $H^1-L^1$ and $L^{\infty}-\textrm{BMO}$ boundedness. Such results can also be concluded from the analysis constituting this paper, but we leave the details to interested readers. Useful hints in this direction can be found, e.g., in~\cite{BCN}.

\section{Kernel estimates and the proofs of Theorems \ref{thm:kerest} and \ref{Sthm:kerest}} \label{sec:ker}

In this section we gather various facts, some of them proved earlier elsewhere, which f\/inally allow us to show Theorems~\ref{thm:kerest} and~\ref{Sthm:kerest}, i.e., the standard estimates for all the relevant kernels. Our approach is based on the technique of proving standard estimates in the context of Laguerre function expansions of convolution type established in~\cite{NoSt2} for the restricted range of the type parameter $\a \in [-1/2,\infty)^d$ and then generalized in~\cite{NoSz} to all admissible $\a \in (-1 , \infty)^d$. Moreover, to treat Lusin area integrals, which are the most complex operators in this paper, we will use an adaptation of the method elaborated in the context of the Dunkl Laplacian in~\cite{CaSz} and having roots in~\cite{TZS3}. We emphasize that in this section all $\a \in (-1,\infty)^d$ are treated in a unif\/ied way, however the restriction $\a \in [-1/2,\infty)^d$ would allow to simplify and shorten the analysis.

\subsection{Preparatory results}

To begin with we prove the following result that allows us to control various derivatives of the auxiliary heat kernels under consideration.

\begin{Lemma} \label{lem:heatEST}
Let $d \ge 1$, $\a \in (-1,\infty)^d$, $n,l,r \in \N^d$, $m \in \N$, $\eta \in \{0,1\}^d$, $\om \in \{-1,1\}^{|n|}$. Then
\begin{gather*}
 \big| \partial_y^r \partial_x^l \partial^m_t \delta_{\eta,n,\om,x}
\mathfrak{G}_t^{\a,\eta,+} (x,y) \big|+ \big| \partial_y^r \partial_x^l \partial^m_t \dersymstate_{\eta,n,x}
\mathbb{G}_t^{\a,\eta,+} (x,y) \big|\\
\qquad{} \lesssim
\sum_{\substack{ \eps, \rho, \xi \in \{0,1\}^d \\
a, b \in \{0,1,2\}^d }} x^{\eta - \rho\eta + 2\eps - a\eps} y^{\eta - \xi \eta + 2\eps - b \eps}
\big(1-\zeta^2\big)^{d + |\a| + |\eta| + 2|\eps|} \\
\qquad\quad{} \times \zeta^{-d - |\a| - |\eta| - 2|\eps| - m- (|r| + |l| + |n|)/2 +(|\rho \eta| + |a \eps| + |\xi \eta| + |b \eps| )/2 } \\
\qquad\quad{} \times \int \sqrt{\ee\bb} \piint ,
\end{gather*}
uniformly in $\zeta \in (0,1)$, $s\in(-1,1)^d$ and $x,y \in \RRR$; here $\zeta=\zeta(t)=\tanh t$.
\end{Lemma}

To prove Lemma~\ref{lem:heatEST} we will use Fa\`a di Bruno's formula for the $N$th derivative, $N \ge 1$, of the composition of two functions, see~\cite{Jo},{\samepage
\begin{gather} \label{Faa}
\partial_{x}^N(g\circ f)(x) = \sum \frac{N!}{p_1! \cdots p_N!} \partial^{p_1+\dots+p_N}
	g\circ f(x) \left( \frac{\partial_x^1 f(x)}{1!}\right)^{p_1} \cdots	\left( \frac{\partial_x^N f(x)}{N!}\right)^{p_N},
\end{gather}
where the summation runs over all $p_1,\ldots,p_N \ge 0$ such that $p_1+2p_2+\dots+N p_N = N$.}

\begin{proof}[Proof of Lemma~\ref{lem:heatEST}] Since $\mathfrak{G}_t^{\a,\eta,+} (x,y)$ and $\mathbb{G}_t^{\a,\eta,+} (x,y)$ have a product structure, an application of Leibniz' rule produces
\begin{gather*}
\partial_y^r \partial_x^l \partial^m_t \delta_{\eta,n,\om,x}
\mathfrak{G}_t^{\a,\eta,+} (x,y)=\sum_{M \in \N^d, \, |M| = m}
c_M \prod_{i = 1}^d \partial_{y_i}^{r_i} \partial_{x_i}^{l_i} \partial^{M_i}_t \delta_{i,\eta_{i},n_i,\om^i,x_i}
\mathfrak{G}_t^{\a_i,\eta_i,+} (x_i,y_i)
\end{gather*}
and a similar identity related to $\mathbb{G}_t^{\a,\eta,+} (x,y)$. Hence we see that to prove Lemma~\ref{lem:heatEST} it suf\/f\/ices to consider the one-dimensional situation. Therefore from now on we assume that $d=1$. We treat each of the two terms in the left-hand side in question separately.

Notice that the Laguerre--Dunkl heat semigroup $\mathfrak{T}_t^\a = \exp(-t\mathfrak{L_{\a}})$, $t >0$, satisf\/ies the heat equation
$\partial_t \mathfrak{T}_t^\a f (x) = - \mathfrak{L_{\a}} \mathfrak{T}_t^\a f (x)$. Thus the Laguerre--Dunkl heat kernel $\mathfrak{G}_t^{\a} (x,y)$ also satisf\/ies this equation with respect to $x$, i.e., we have
\begin{gather*}
\partial_t \mathfrak{G}_t^{\a} (x,y) = (T_x^\a)^2 \mathfrak{G}_t^{\a} (x,y) - x^2 \mathfrak{G}_t^{\a} (x,y),
\end{gather*}
as can be easily checked by using the identity $\mathfrak{L_{\a}} = -(T^\a)^2 + x^2$. Denote for brevity $\partial_{\eta,x} = \partial_{1,\eta,x}$ (see~\eqref{pp}) and observe that for each $\eta \in \{ 0,1 \}$ the functions
\begin{gather*}
\partial_t \mathfrak{G}_t^{\a,\eta,+} (x,y) \qquad \textrm{and} \qquad \partial_{\overline{\eta+1},x} \partial_{\eta,x} \mathfrak{G}_t^{\a,\eta,+} (x,y) - x^2
\mathfrak{G}_t^{\a,\eta,+} (x,y)
\end{gather*}
are $\eta$-symmetric with respect to $x$, and with the aid of decomposition \eqref{GD} they are both $\eta$-symmetric components of $2 \partial_t \mathfrak{G}_t^{\a} (x,y)$. By the uniqueness of $\eta$-symmetric component we get
\begin{gather*}
\partial_{\overline{\eta+1},x} \partial_{\eta,x} \mathfrak{G}_t^{\a,\eta,+} (x,y) = \partial_t \mathfrak{G}_t^{\a,\eta,+} (x,y)+ x^2 \mathfrak{G}_t^{\a,\eta,+} (x,y).
\end{gather*}
Using this identity and proceeding inductively we infer that
\begin{gather*}
\delta_{\eta,n,\om,x} \mathfrak{G}_t^{\a,\eta,+} (x,y)= \sum_{\substack{q = 0,1 \\ 2p + q \le n}}
P_{n,p,q,\eta,\om} (x) \partial_t^p \partial^q_{\eta,x} \mathfrak{G}_t^{\a,\eta,+} (x,y),
\end{gather*}
where $P_{n,p,q,\eta,\om}$ is a $(\overline{n-q})$-symmetric polynomial of degree at most $n - 2p - q$; here and later on we use the natural convention that $\partial^q_{\eta,x} = \id$ for $q = 0$. Combining this with Leibniz' rule and \cite[Lemma~A.3(d)]{NoSz} we see that our task reduces to showing the estimate
\begin{gather}
\big| \partial_y^R \partial_x^L \partial_t^M \partial_{\eta,x}^q \mathfrak{G}_t^{\a,\eta,+} (x,y) \big|
 \lesssim
\sum_{\substack{ \eps, \rho, \xi = 0,1 \\
a, b = 0,1,2 }}
x^{\eta - \rho\eta + 2\eps - a\eps} y^{\eta - \xi \eta + 2\eps - b \eps}
\big(1-\zeta^2\big)^{1 + \a + \eta + 2\eps}\label{est2}
\\ \nonumber
\hphantom{\big| \partial_y^R \partial_x^L \partial_t^M \partial_{\eta,x}^q \mathfrak{G}_t^{\a,\eta,+} (x,y) \big|\lesssim}{}
 \times
\zeta^{ -1 - \a - \eta - 2\eps - M - (R + L + q)/2 +
(\rho \eta + a \eps + \xi \eta + b \eps )/2 } \\
\hphantom{\big| \partial_y^R \partial_x^L \partial_t^M \partial_{\eta,x}^q \mathfrak{G}_t^{\a,\eta,+} (x,y) \big|\lesssim}{}
\times
	\int \big( \e\bb \big)^{3/4} \piinta,\nonumber
\end{gather}
uniformly in $\z \in (0,1)$, $s \in (-1,1)$ and $x,y \in \R_+$; here $R,L,M \in \N$, $\eta,q \in \{0, 1 \}$ are f\/ixed and $\z = \z(t) = \tanh t$.
Combining~\eqref{def:aheatD} with the representation formula~\eqref{IRL} of $G_t^{\a} (x,y)$ our task can be further reduced to proving that
\begin{gather} \nonumber
 \big| \partial_y^R \partial_x^L \partial_t^M \big[ \big(1 - \z^2\big)^{W_1} \z^{W_2}
x^{\eta_1 + 2\eps } y^{\eta_2 + 2\eps } \e \bb
 \big] \big| \\
 \qquad{} \lesssim \sum_{\substack{ \rho, \xi = 0,1 \\
a, b = 0,1,2 }}
x^{\eta_1 - \rho\eta_1 + 2\eps - a\eps} y^{\eta_2 - \xi \eta_2 + 2\eps - b \eps}
\big(1 - \z^2\big)^{W_1}\label{est1} \\
 \qquad \quad {}\times \z^{W_2 - M - (R + L)/2 +(\rho \eta_1 + a \eps + \xi \eta_2 + b \eps )/2 }\big( \e\bb \big)^{3/4}, \nonumber
\end{gather}
uniformly in $\z \in (0,1)$, $s \in (-1,1)$ and $x,y \in \R_+$; here $R,L,M \in \N$, $W_1,W_2 \in \R$ and $\eta_1, \eta_2, \eps \in \{0, 1 \}$ are f\/ixed, and $\z = \z(t) = \tanh t$. We now focus on proving~\eqref{est1}.

By Leibniz' rule for every $L \in \N$ and $\eta_1, \eps \in \{0, 1 \}$ we have
\begin{gather*}
\partial_x^L \big[ x^{\eta_1 + 2 \eps} f \big]=
\sum_{\substack{ \rho = 0,1 \\
a = 0,1,2 }}
c_{L, \eta_1, \eps, \rho, a } \, \chi_{ \{ L \ge \rho \eta_1 + a \eps \} }
x^{\eta_1 - \rho\eta_1 + 2\eps - a\eps}\partial_x^{L - \rho\eta_1 - a\eps} f,
\end{gather*}
where $c_{L, \eta_1, \eps, \rho, a } \in \R$ are constants. This leads to the equation
\begin{gather*}
\partial_y^R \partial_x^L \big[ x^{\eta_1 + 2 \eps} y^{\eta_2 + 2\eps } f \big] =
\sum_{\substack{ \rho, \xi = 0,1 \\
a, b = 0,1,2 }}
c_{L, R, \eta_1, \eta_2, \eps, \rho, \xi, a, b }
 \chi_{ \{ L \ge \rho \eta_1 + a \eps, \, R \ge \xi \eta_2 + b \eps \} } \\
\hphantom{\partial_y^R \partial_x^L \big[ x^{\eta_1 + 2 \eps} y^{\eta_2 + 2\eps } f \big]=}{} \times
x^{\eta_1 - \rho\eta_1 + 2\eps - a\eps}
y^{\eta_2 - \xi \eta_2 + 2\eps - b \eps}
\partial_y^{R - \xi \eta_2 - b \eps} \partial_x^{L - \rho\eta_1 - a\eps} f,
\end{gather*}
where $c_{L, R, \eta_1, \eta_2, \eps, \rho, \xi, a, b } \in \R$ are constants; here $R,L \in \N$ and $\eta_1, \eta_2, \eps \in \{0, 1 \}$ are f\/ixed. Further, denoting
\begin{gather*}
F = F(\z,q_{\pm}) = \log\e(\z,q_{\pm})= - \frac{1}{4\z}q_{+} - \frac{\z}{4}q_{-},
\end{gather*}
and using Fa\`a di Bruno's formula \eqref{Faa} (notice that $\partial_x^3 F = \partial_x^4 F = \dots = 0$) we arrive at
\begin{gather*}
\partial_x^L \e \bb = \sum_{L_1 + 2 L_2 = L} c_{L_1,L_2} (\partial_x F)^{L_1} (\partial_x^2 F)^{L_2} \e \bb,
\end{gather*}
where $c_{L_1,L_2} \in \R$ are constants; observe that this formula works also for $L=0$. By Leibniz' rule and another application of \eqref{Faa} we obtain
\begin{gather*}
 \partial_y^R \partial_x^L \e \bb \\
 \qquad {}= \sum_{p_1 + p_2 + 2p_3 + 2p_4 = L + R} c_{L,R,p_1, \ldots, p_4}
(\partial_x F)^{p_1} (\partial_y F)^{p_2} (\partial_x^2 F)^{p_3}(\partial_x \partial_y F)^{p_4} \e \bb,
\end{gather*}
where $c_{L,R,p_1, \ldots, p_4} \in \R$ are constants, possibly zero; here we used the identities $\partial_x^2 F = \partial_y^2 F$ and $\partial_y^R [ (\partial_x F)^{L_1} ] = \chi_{ \{ L_1 \ge R \} } c_{L_1,R} (\partial_x F)^{L_1 - R} (\partial_x \partial_y F)^{R}$ for some constants $c_{L_1,R} \in \R$.
These facts altogether give the identity
\begin{gather*}
 \partial_y^R \partial_x^L \partial_t^M \big[ \big(1 - \z^2\big)^{W_1} \z^{W_2}
x^{\eta_1 + 2\eps } y^{\eta_2 + 2\eps } \e \bb \big] \\
 \qquad{} = \sum_{\substack{ \rho, \xi = 0,1 \\
a, b = 0,1,2 }}\chi_{ \{ L \ge \rho \eta_1 + a \eps, \, R \ge \xi \eta_2 + b \eps \} }
\bigg( x^{\eta_1 - \rho\eta_1 + 2\eps - a\eps} y^{\eta_2 - \xi \eta_2 + 2\eps - b \eps} \\
 \qquad \quad{} \times
\sum_{p_1 + p_2 + 2p_3 + 2p_4 = L + R - \rho \eta_1 - a \eps - \xi \eta_2 - b \eps}
c_{L,R,\eta_1,\eta_2, \eps, \rho, \xi, a, b,p_1, \ldots, p_4} \\
 \qquad \quad {}\times
\partial_t^M \big[ \big(1 - \z^2\big)^{W_1} \z^{W_2} (\partial_x F)^{p_1} (\partial_y F)^{p_2} (\partial_x^2 F)^{p_3}
(\partial_x \partial_y F)^{p_4} \e \bb \big] \bigg),
\end{gather*}
where $c_{L,R,\eta_1,\eta_2, \eps, \rho, \xi, a, b,p_1, \ldots, p_4} \in \R$ are constants. Hence to prove \eqref{est1} it is enough to check that
\begin{gather} \nonumber
 \big| \partial_t^M \big[ \big(1 - \z^2\big)^{W_1} \z^{W_2}
(\partial_x F)^{p_1} (\partial_y F)^{p_2} (\partial_x^2 F)^{p_3}
(\partial_x \partial_y F)^{p_4} \e \bb \big] \big| \\ \label{est3}
 \qquad {}\lesssim
\big(1 - \z^2\big)^{W_1} \z^{W_2 - M - p_1/2 - p_2/2 - p_3 - p_4}\big( \e \bb \big)^{3/4},
\end{gather}
uniformly in $\z \in (0,1)$, $s \in (-1,1)$ and $x,y \in \R_+$; here $M, p_1, \ldots, p_4 \in \N$ and $W_1,W_2 \in \R$ are f\/ixed, and $\z = \z(t) = \tanh t$. We now justify this estimate.

Using \cite[equation~(A.3)]{NoSz} and \cite[Lemma~A.3(a)]{NoSz} we see that for $M \in \N$ and $S_1,S_2 \in \R$ f\/ixed we have
\begin{gather} \label{est4}
\big| \partial_{\z}^M \big[ \big(1 - \z^2\big)^{S_1} \z^{S_2}\e \bb \big] \big| \lesssim \big(1 - \z^2\big)^{S_1 - M} \z^{S_2 - M} \big( \e \bb \big)^{7/8},
\end{gather}
uniformly in $\z \in (0,1)$, $s \in (-1,1)$ and $x,y \in \R_+$. Further, since
\begin{alignat*}{4}
& \partial_x F = - \frac{1}{2 \z} \Psi_{+} - \frac{\z}{2 } \Psi_{-},
\qquad &&
\partial_x^2 F = - \frac{1}{2 \z} - \frac{\z}{2 },
\qquad&&
\Psi_{\pm} = \Psi_{\pm}(x,y,s) = x \pm ys,& \\
& \partial_y F = - \frac{1}{2 \z} \Phi_{+} - \frac{\z}{2 } \Phi_{-},
\qquad &&
\partial_x \partial_y F = s \left({-} \frac{1}{2 \z} + \frac{\z}{2 } \right),
\qquad &&
\Phi_{\pm} = \Phi_{\pm}(x,y,s) = y \pm xs,&
\end{alignat*}
by Newton's formula we get
\begin{gather*}
(\partial_x F)^{p_1} (\partial_y F)^{p_2} (\partial_x^2 F)^{p_3}(\partial_x \partial_y F)^{p_4} \\
 \qquad {} =s^{p_4} \sum_{j = - p_3 - p_4}^{p_3 + p_4} \sum_{k_1 = 0}^{p_1}
\sum_{k_2 = 0}^{p_2} c_{p_1, \ldots, p_4, j, k_1, k_2} \z^j
\left( \frac{\Psi_+}{\z} \right)^{k_1}(\z \Psi_-)^{p_1 - k_1}
\left( \frac{\Phi_+}{\z} \right)^{k_2} (\z \Phi_-)^{p_2 - k_2}.
\end{gather*}
Consequently, using \eqref{est4} (with $S_1 = W_1$, $S_2 = W_2 + j + p_1 -2k_1 + p_2 -2k_2$) and \cite[Lemma~A.3(b) and (c)]{NoSz} we arrive at
\begin{gather*}
 \big| \partial_{\z}^M \big[ \big(1 - \z^2\big)^{W_1} \z^{W_2}
(\partial_x F)^{p_1} (\partial_y F)^{p_2} (\partial_x^2 F)^{p_3}
(\partial_x \partial_y F)^{p_4} \e \bb \big] \big| \\
 \qquad{} \lesssim
 \big(1 - \z^2\big)^{W_1 - M} \z^{W_2 - M - p_1/2 - p_2/2 - p_3 - p_4} \big( \e \bb \big)^{3/4},
\end{gather*}
uniformly in $\z \in (0,1)$, $s \in (-1,1)$ and $x,y \in \R_+$; here $M, p_1, \ldots, p_4 \in \N$ and $W_1,W_2 \in \R$ are f\/ixed. Finally, an application of Fa\`a di Bruno's formula \eqref{Faa} and the identity \cite[equation~(A.2)]{NoSz} leads us to \eqref{est3}, and the desired estimate connected with $\mathfrak{G}_t^{\a,\eta,+} (x,y)$ follows.

It remains to deal with the estimate in question for the kernel $\mathbb{G}_t^{\a,\eta,+}(x,y)$. Proceeding in an analogous way as at the beginning of the proof we obtain the formula $\partial_t \mathbb{G}_t^{\a} (x,y) = (\mathbb{D}_x^2 - \lambda_0^{\a}) \mathbb{G}_t^{\a} (x,y)$ and, consequently, for symmetry reasons,
\begin{gather*}
\dersym_{\eta,2,x} \mathbb{G}_t^{\a,\eta,+} (x,y) = (\partial_t + \lambda_0^{\a}) \mathbb{G}_t^{\a,\eta,+} (x,y).
\end{gather*}
Iterating the latter identity we infer that
\begin{gather*}
\dersym_{\eta,n,x} \mathbb{G}_t^{\a,\eta,+} (x,y) = (\partial_t + \lambda_0^{\a})^{\lfloor n/2 \rfloor} \dersym_{\eta,\overline{n},x}
\mathbb{G}_t^{\a,\eta,+} (x,y).
\end{gather*}
Since $\dersym_{\eta,\overline{n},x} = \partial_{\eta,x}^{\overline{n}} + \overline{n} (-1)^{\eta} x$ and $\mathbb{G}_t^{\a,\eta,+} (x,y) = e^{-2 \eta t} \mathfrak{G}_t^{\a,\eta,+} (x,y)$,
it is easy to check that
\begin{gather*}
\dersym_{\eta,n,x} \mathbb{G}_t^{\a,\eta,+} (x,y)=e^{-2 \eta t}\sum_{\substack{q = 0,1 \\ 2p + q \le n}}
P_{n,p,q,\eta} (x) \partial_t^p \partial_{\eta,x}^q \mathfrak{G}_t^{\a,\eta,+} (x,y),
\end{gather*}
where $P_{n,p,q,\eta}$ is a $(\overline{n-q})$-symmetric polynomial of degree at most $n - 2p - q$. Now applying Leibniz' rule (to the variables $x$ and $t$) and then using sequently~\eqref{est2} and \cite[Lemma~A.3(d)]{NoSz} we obtain the required bound for the quantity related to $\mathbb{G}_t^{\a,\eta,+} (x,y)$.

This f\/inishes the whole reasoning justifying Lemma~\ref{lem:heatEST}.
\end{proof}

The next lemma is an essence of the method of proving standard estimates presented in this paper. It provides a link from the estimates obtained in Lemma~\ref{lem:heatEST} to the standard estimates for the space $(\RRR,\mu_{\a}^+,\|\cdot\|)$. We note that only the values $p \in \{1,2,\infty\}$ will be needed for our purposes. However, other values of $p$ are also important, for instance in connection with more general square functions introduced in Section~\ref{ssec:fur}. The lemma below is proved in much the same fashion as \cite[Lemma~2.6]{NoSz}, hence we omit the details.

\begin{Lemma}\label{lem:bridge}
Assume that $\a \in (-1 , \infty)^d$, $1 \le p \le \infty$, $W \in \mathbb{R}$, $C>0$, $\eps, \eta, \rho, \xi \in \{ 0,1 \}^d$ and $a,b \in \{ 0,1,2 \}^d$ are fixed. Further, let $\tau \in \N^d$ be such that $\tau \le \eta - \rho \eta + 2\eps - a \eps$ and let $D$ be a fixed constant satisfying $D < d + |\a|$. Given $u \ge 0$, we consider the function $\Upsilon_u \colon \RRR \times \RRR \times (0,1) \to \mathbb{R}$
defined by
 \begin{gather*}
 \Upsilon_{u}(x,y,\z(t)) = x^{\eta - \rho \eta + 2\eps- a \eps - \tau}
 y^{\eta- \xi \eta + 2\eps- b \eps }	\big( \loo (\z) \big)^{|\tau|/2} \big(1 - \z^2\big)^{d + |\a| + |\eta| + 2 |\eps| - D} \\
\hphantom{\Upsilon_{u}(x,y,\z(t)) =}{} \times
 \z^{ -d - |\a| - |\eta| - 2|\eps| + (|\rho \eta| + |a \eps| + |\xi \eta| + |b \eps|)\slash 2 - W\slash p-u\slash 2} \\
\hphantom{\Upsilon_{u}(x,y,\z(t)) =}{}
 \times \int \big( \ee \bb \big)^C \piint,
 \end{gather*}
 where $W \slash p = 0$ for $p=\infty$. Then $\Upsilon_u$ satisfies the integral estimate
 \begin{gather*}
 \big\|\Upsilon_{u}\big(x,y,\z(t)\big)\big\|_{L^{p}(\R_+,t^{W-1}dt)} \lesssim \frac{1}{||x-y||^{u}} \frac{1}{\mu^+_{\a}(B(x,||y-x||))},
 \end{gather*}
 uniformly in $x,y \in \RRR$, $x\neq y$.
\end{Lemma}

The following remark will be useful when estimating the kernels associated with multipliers of Laplace--Stieltjes type.

\begin{Remark}\label{rem:bridgeL-S}
The norm estimate in Lemma~\ref{lem:bridge} still holds true if $D = d + |\a|$, $p=\infty$ and $\tau = \boldsymbol{0}$.
\end{Remark}

\begin{Lemma}[{\cite[Lemma~4.3]{TZS3}}] \label{lem:theta}
Let $x,y,z\in\RRR$ and $s \in (-1,1)^d$. Then
\begin{gather*}
\frac{1}{4} q_{\pm}(x,y,s) \le q_{\pm}(z,y,s) \le 4 q_{\pm}(x,y,s),
\end{gather*}
provided that $\|x-y\|>2\|x-z\|$. Similarly, if $\|x-y\|>2\|y-z\|$ then
\begin{gather*}
\frac{1}{4} q_{\pm}(x,y,s) \le q_{\pm}(x,z,s) \le 4 q_{\pm}(x,y,s).
\end{gather*}
\end{Lemma}

\begin{Lemma}[{\cite[Lemma~4.5]{TZS3}, \cite[Lemma~4.4]{CaSz}}] \label{lem:double}
Let $\a \in (-1 ,\infty)^d$ and $\gamma \in \mathbb{R}$ be fixed. We have
\begin{gather*}
\left( \frac{1}{\|z-y\|} \right)^{\gamma} \frac{1}{\mu^+_{\a}(B(z,\|z-y\|))}
\simeq\left( \frac{1}{\|x-y\|} \right)^{\gamma}\frac{1}{\mu^+_{\a}(B(x,\|x-y\|))}
\end{gather*}
on the set $\{(x,y,z) \in \RRR \times \RRR \times \RRR\colon \|x-y\|>2\|x-z\|\}$.
\end{Lemma}

The next two lemmas will be crucial when dealing with the kernels associated with the Lusin area integrals.

\begin{Lemma}[{\cite[Lemma~4.7]{TZS3}}] \label{lem:qz}
Let $x,y\in\RRR$, $z\in\RR$, $s\in(-1,1)^d$. Then
\begin{gather*}
q_{\pm}(x+z,y,s) \geq \frac{1}{2}q_{\pm}(x,y,s) - \|z\|^{2}.
\end{gather*}
\end{Lemma}

The result below is a combination of \cite[Lemmas 4.6--4.8]{CaSz}.

\begin{Lemma}[{\cite[Lemmas 4.6--4.8]{CaSz}}] \label{lem:intLusin}
Let $\a \in (-1,\infty)^d$ be fixed. Then there exists $\gamma=\gamma(\a) \in (0,1/2]$ such that
\begin{itemize}\itemsep=0pt
\item[$(a)$]
\begin{gather*}
\int_{\|z\|<\sqrt{t}} \Xi_{\a}(x,z,t) \chi_{ \{ x+z \in \RRR \} } \, dz \simeq 1, \qquad x \in \RRR, \quad t>0.
\end{gather*}
\item[$(b)$]
\begin{gather*}
 \int_{\|z\|<\sqrt{t}} \chi_{ \{ x+z, \, x'+z \in \RRR \} }\big| \sqrt{ \Xi_\a(x,z,t) } - \sqrt{ \Xi_\a(x',z,t) }\big|^2 \, dz \\
 \qquad{} \lesssim\left( \frac{\|x-x'\|^2}{t} \right)^{\gamma} \lesssim\left( \frac{\|x-x'\|^2}{\z} \right)^{\gamma},
\end{gather*}
uniformly in $x,x' \in \RRR$ and $t>0$; here $\zeta=\zeta(t)=\tanh t$.
\item[$(c)$]
\begin{gather*}
\int_{\|z\|<\sqrt{t}} \chi_{ \{ x+z \in \RRR, \, x'+z \notin \RRR \} } \Xi_{\a}(x,z,t) \, dz\lesssim
\left( \frac{\|x-x'\|^2}{t} \right)^{\gamma}\lesssim\left( \frac{\|x-x'\|^2}{\z} \right)^{\gamma},
\end{gather*}
uniformly in $x,x' \in \RRR$ and $t>0$; here $\zeta=\zeta(t)=\tanh t$.
\end{itemize}
Moreover, in items $(b)$ and $(c)$ one can take any $\gamma \in (0,1/2]$ satisfying $\gamma < \min\limits_{1\le i \le d}(\a_i +1)$.
\end{Lemma}

\subsection{Proofs of the standard estimates}

In the proof of Theorem~\ref{thm:kerest} we tacitly assume that passing with the dif\/ferentiation in~$x_i$ and~$y_i$ under integrals against~$dt$ and~$d\nu(t)$ is legitimate. Actually, such manipulations can easily be justif\/ied with the aid of the dominated convergence theorem and the estimates obtained in Lemma~\ref{lem:heatEST} and along the proof of Theorem~\ref{thm:kerest}.

\begin{proof}[Proof of Theorem~\ref{thm:kerest}] We will treat each of the kernels separately.

\textbf{The case of $\mathfrak{U}^{\alpha, \eta, +}(x,y)$.} The growth condition \eqref{gr} is a direct consequence of Lemma~\ref{lem:heatEST} (applied with $r = l = n = \boldsymbol{0}$ and $m = 0$) and Lemma~\ref{lem:bridge} (specif\/ied to $p = \infty$, $W = 1$, $C = 1/2$, $\tau=\boldsymbol{0}$, $D = u = 0$).

To verify the smoothness estimates, for symmetry reasons, it suf\/f\/ices to show only \eqref{sm1}. By the mean value theorem we have
\begin{gather*}
| \mathfrak{G}_t^{\a,\eta,+} (x,y) - \mathfrak{G}_t^{\a,\eta,+} (x',y) |\le
\|x - x'\| \left\| \nabla_{\!x} \mathfrak{G}_t^{\a,\eta,+} (x,y) \big|_{x=\t} \right\|,
\end{gather*}
where $\t = \t(t,x,x',y)$ is a convex combination of $x$ and $x'$. Now, applying sequently Lemma~\ref{lem:heatEST} (with $r = n = \boldsymbol{0}$, $m = 0$ and $l = e_i$, $i=1, \ldots, d$),
the inequalities
\begin{gather} \label{ineq1}
 \t \le x \vee x' , \qquad \|x - \t\| \le \|x-x'\| , \qquad \|x - x \vee x'\| \le \|x - x'\|,
\end{gather}
and Lemma~\ref{lem:theta} twice (f\/irst with $z = \t$ and then with $z = x \vee x'$), we obtain
\begin{gather*}
| \mathfrak{G}_t^{\a,\eta,+} (x,y) - \mathfrak{G}_t^{\a,\eta,+} (x',y) | \\
\qquad {} \lesssim \|x - x'\| \sum_{\substack{ \eps, \rho, \xi \in \{0,1\}^d \\
a, b \in \{0,1,2\}^d }} (x \vee x')^{\eta - \rho\eta + 2\eps - a\eps} y^{\eta - \xi \eta + 2\eps - b \eps}
\big(1-\zeta^2\big)^{d + |\a| + |\eta| + 2|\eps|} \\
\qquad\quad{} \times \zeta^{-d - |\a| - |\eta| - 2|\eps| -1/2 + (|\rho \eta| + |a \eps| + |\xi \eta| + |b \eps| )/2 } \\
 \qquad \quad {}\times 	\int \big(\e\eee\big)^{1/32} \, \piint ,
\end{gather*}
provided that $\|x-y\|> 2\|x-x'\|$. This, together with Lemma~\ref{lem:bridge} (with $p = \infty$, $W = 1$, $C = 1/32$, $\tau = \boldsymbol{0}$, $D = 0$
and $u = 1$) and Lemma~\ref{lem:double} (taken with $\gamma=1$ and $z=x \vee x'$), produces the required estimate.

\textbf{The case of $\mathfrak{R}_{n,\om}^{\alpha,\eta,+}(x,y)$.}
The growth bound is an easy consequence of
Lemma~\ref{lem:heatEST} (specif\/ied to $r = l = \boldsymbol{0}$, $m = 0$)
and Lemma~\ref{lem:bridge} (with $p = 1$, $W = |n|/2$, $C = 1/2$, $\tau = \boldsymbol{0}$, $D = u = 0$).

To prove the gradient bound \eqref{grad} it is enough to check that
\begin{gather*}
\left\| \big\| \nabla_{\!x,y} \delta_{\eta,n,\om,x} \mathfrak{G}_t^{\a,\eta,+} (x,y) \big\| \right\|_{L^1(\R_+,t^{|n|/2 - 1}dt)} \\
 \qquad {} \lesssim
\frac{1}{\|x-y\|\mu^+_{\a}(B(x,\|x-y\|))}, \qquad x,y \in \RRR, \qquad x \ne y.
\end{gather*}
This, however, follows by combining Lemma~\ref{lem:heatEST} (taken with $m = 0$ and either $r = e_i$, $l = \boldsymbol{0}$ or $r = \boldsymbol{0}$, $l =e_i$, $i = 1, \ldots, d$) with Lemma~\ref{lem:bridge} (specif\/ied to $p = 1$, $W = |n|/2$, $C = 1/2$, $\tau = \boldsymbol{0}$, $D = 0$ and $u = 1$).

\textbf{The case of $\mathfrak{K}^{\alpha, \eta, +}_{\psi}(x,y)$.} Since $\psi$ is bounded, the growth condition is a straightforward consequence of Lemma~\ref{lem:heatEST} (with $r = l = n = \boldsymbol{0}$ and $m = 1$) and Lemma~\ref{lem:bridge} (selecting $p = 1$, $W = 1$, $C = 1/2$, $\tau = \boldsymbol{0}$, $D = u = 0$).

Next we pass to proving the gradient estimate \eqref{grad}. Once again, using the boundedness of~$\psi$ and for symmetry reasons, it is enough to verify that
\begin{gather*}
\left\| \big\| \nabla_{\!x} \partial_t \mathfrak{G}_t^{\a,\eta,+} (x,y) \big\| \right\|_{L^1(\R_+,dt)}\lesssim
\frac{1}{\|x-y\|\mu^+_{\a}(B(x,\|x-y\|))},\qquad x,y \in \RRR, \qquad x \ne y.
\end{gather*}
Applying Lemma~\ref{lem:heatEST} (with $r = n = \boldsymbol{0}$, $m = 1$ and $l =e_i$, $i = 1, \ldots, d$) together with Lemma~\ref{lem:bridge} (choosing $p = 1$, $W = 1$, $C = 1/2$, $\tau = \boldsymbol{0}$, $D = 0$ and $u = 1$) we get the asserted estimate.

\textbf{The case of $\mathfrak{K}^{\alpha, \eta, +}_{\nu}(x,y)$.} By the assumption \eqref{assum} the growth bound is reduced to showing that
\begin{gather*}
\big\|e^{t \lambda_0^\a} \mathfrak{G}_t^{\a,\eta,+} (x,y) \big\|_{L^\infty(\R_+,dt)}
\lesssim\frac{1}{\mu^+_{\a}(B(x,\|x-y\|))}, \qquad x,y \in \RRR, \qquad x \ne y.
\end{gather*}
This, however, follows from Lemma~\ref{lem:heatEST} (applied with $r = l = n = \boldsymbol{0}$, $m = 0$) and Remark~\ref{rem:bridgeL-S} (with $W = 1$, $C = 1/2$, $u = 0$).

In order to prove the gradient estimate \eqref{grad}, for symmetry reasons, it suf\/f\/ices to verify that
\begin{gather*}
\left\| e^{t \lambda_0^\a} \big\| \nabla_{\!x} \mathfrak{G}_t^{\a,\eta,+} (x,y) \big\| \right\|_{L^\infty(\R_+,dt)}
\lesssim\frac{1}{\|x-y\|\mu^+_{\a}(B(x,\|x-y\|))},\qquad x,y \in \RRR, \qquad x \ne y.
\end{gather*}
Combining Lemma~\ref{lem:heatEST} (taken with $r = n = \boldsymbol{0}$, $m = 0$ and $l =e_i$, $i = 1, \ldots, d$) with Remark~\ref{rem:bridgeL-S} (with $W = 1$, $C = 1/2$, $u = 1$) we get the required bound.

\textbf{The case of $\mathfrak{H}^{\alpha, \eta, +}_{n,m,\om}(x,y)$.} The growth condition is a direct consequence of Lemma~\ref{lem:heatEST} (specif\/ied to $r = l = \boldsymbol{0}$) and Lemma~\ref{lem:bridge} (taken with $p = 2$, $W = |n| + 2m$, $C = 1/2$, $\tau = \boldsymbol{0}$, $D = u = 0$).

{\allowdisplaybreaks We pass to proving the smoothness estimates. We focus on showing~\eqref{sm1}, the other bound can be justif\/ied in a similar way. Using sequently the mean value theorem, Lemma~\ref{lem:heatEST} (with either $r = e_i$, $l = \boldsymbol{0}$ or $r = \boldsymbol{0}$, $l =e_i$, $i = 1, \ldots, d$), the inequalities~\eqref{ineq1} and Lemma~\ref{lem:theta} twice (f\/irst with $z = \t$ and then with $z = x \vee x'$) we see that
\begin{gather*}
 \Big| \partial^m_t \delta_{\eta,n,\om,x} \mathfrak{G}_t^{\a,\eta,+} (x,y) - \partial^m_t
\delta_{\eta,n,\om,x} \mathfrak{G}_t^{\a,\eta,+} (x,y) \big|_{x=x'} \Big| \\
 \qquad{} \le \|x - x'\| \left\| \nabla_{\!x,y} \partial^m_t \delta_{\eta,n,\om,x} \mathfrak{G}_t^{\a,\eta,+} (x,y) \big|_{x=\t} \right\| \\
 \qquad{} \lesssim \|x - x'\|\!\sum_{\substack{ \eps, \rho, \xi \in \{0,1\}^d \\ a, b \in \{0,1,2\}^d }}\!
(x \vee x')^{\eta - \rho\eta + 2\eps - a\eps} y^{\eta - \xi \eta + 2\eps - b \eps}
\big(1-\zeta^2\big)^{d + |\a| + |\eta| + 2|\eps|} \zeta^{-d - |\a| - |\eta| - 2|\eps|}\\
\qquad \quad {}\times\zeta^{- m - |n|/2 - 1/2 +(|\rho \eta| + |a \eps| + |\xi \eta| + |b \eps| )/2} \int \big(\e\eee\big)^{1/32} \piint ,
\end{gather*}
provided that $\|x-y\|> 2\|x-x'\|$; here $\t = \t (t,x,x',y)$ is a convex combination of $x$ and $x'$. Now an application of Lemma~\ref{lem:bridge} (choosing $p = 2$, $W = |n| + 2m$, $C = 1/32$, $\tau = \boldsymbol{0}$, $D = 0$, $u = 1$) and then Lemma~\ref{lem:double} (with $\gamma=1$ and $z=x \vee x'$) produces the required bound.}

\textbf{The case of $\mathfrak{S}^{\alpha, \eta, +}_{n,m,\om}(x,y)$.} We f\/irst deal with the growth estimate. Fix a constant $0 < D \le 1/4$ such that $D < d + |\a|$. We show that
\begin{gather}\label{ineq3}
\big( \e \zz \big)^{1/2} \lesssim \big(1 - \z^2\big)^{-D} \big( \e \bbb \big)^{D},
\end{gather}
uniformly in $x,y \in \RRR$, $s \in (-1,1)^d$ and $(z,t) \in A$; here and later on $\zeta=\zeta(t)=\tanh t$. Indeed, using Lemma~\ref{lem:qz} we obtain
\begin{gather*}
\big( \e \zz \big)^{1/2}\le\big( \e \zz \big)^{2D} \\
\hphantom{\big( \e \zz \big)^{1/2}}{} \le \big( \e \bbb \big)^{D} \exp\left( D \left( \frac{1}{4\z} + \frac{\z}{4} \right) \lo (\z) \right),
\end{gather*}
provided that $x,y \in \RRR$, $s \in (-1,1)^d$ and $(z,t) \in A$. Now a simple analysis of the second factor in the last expression above gives us~\eqref{ineq3}.

Taking into account Lemma~\ref{lem:heatEST} (specif\/ied to $r = l = \boldsymbol{0}$), \eqref{ineq3} and the estimate
\begin{gather}\label{ineq2}
|(x+z)^\kappa| \le (x+\sqrt{t}\mathbf{1})^\kappa \lesssim \sum_{\boldsymbol{0}\le \tau \le \kappa} x^{\kappa - \tau} \big( \lo (\z) \big)^{|\tau|/2},
\qquad x\in\RRR, \qquad (z,t) \in A,
\end{gather}
where $\kappa \in \mathbb{N}^d$ is f\/ixed, we get
\begin{gather} \nonumber
 \Big| \partial_t^m \delta_{\eta, n, \om, \mathbf{x}} \mathfrak{G}_{t}^{\a,\eta,+}(\mathbf{x},y) \big|_{\mathbf{x} = x+z} \Big| \\
 \qquad{} \lesssim \sum_{\substack{ \eps, \rho, \xi \in \{0,1\}^d \\ a, b \in \{0,1,2\}^d }}
\sum_{\boldsymbol{0} \le \tau \le \eta - \rho\eta + 2\eps - a\eps } x^{\eta - \rho\eta + 2\eps - a\eps - \tau } y^{\eta - \xi \eta + 2\eps - b \eps}
\big( \lo (\z) \big)^{|\tau|/2} \label{ineq4}\\ \nonumber
\qquad\quad{} \times \big(1-\zeta^2\big)^{d + |\a| + |\eta| + 2|\eps| - D}
\zeta^{-d - |\a| - |\eta| - 2|\eps| - m - |n|/2 + (|\rho \eta| + |a \eps| + |\xi \eta| + |b \eps| )/2 } \\
 \qquad \quad {} \times 	\int \big( \e\bbb \big)^D \piint , \nonumber
\end{gather}
for $x,y \in \RRR$ and $(z,t) \in A$ such that $x+z \in \RRR$. Since the right-hand side above is independent of~$z$, an application of Lemma~\ref{lem:intLusin}(a) and then Lemma~\ref{lem:bridge} (specif\/ied to $p = 2$, $W = |n| + 2m$, $C = D$ and $u = 0$) leads to the desired conclusion.

Next we verify the f\/irst smoothness condition. Precisely, we will show~\eqref{sm1} with any f\/ixed $\gamma \in (0,1/2]$ satisfying $\gamma < \min\limits_{1\le i \le d} (\a_i + 1)$. In what follows it is natural to split the region of integration~$A$ into four subsets, depending on whether $x+z$, $x'+z$ belong to $\RRR$ or not. Let
\begin{gather*}
 A_1 = \big\{ (z,t) \in A \colon x+z \in \RRR,\, x'+z \in \RRR \big\},\\
 A_2 = \big\{ (z,t) \in A \colon x+z \in \RRR, \, x'+z \notin \RRR \big\},\\
 A_3 = \big\{ (z,t) \in A \colon x+z \notin \RRR, \, x'+z \in \RRR \big\},\\
 A_4 = \big\{ (z,t) \in A \colon x+z \notin \RRR, \, x'+z \notin \RRR \big\}.
\end{gather*}
Since in case of $A_4$ there is nothing to do and the case of $A_3$ is analogous to $A_2$ (the only dif\/ference is that at the end of reasoning related to $A_3$ one should use Lemma~\ref{lem:double}), we analyze only the two essential cases.

{\bf Case 1:} \textbf{The norm related to $L^2(A_{1},t^{|n| + 2m - 1}dzdt)$.} By the triangle inequality
\begin{gather*}
\left| \partial_t^m \delta_{\eta, n, \om, \mathbf{x}} \mathfrak{G}_{t}^{\a,\eta,+}(\mathbf{x},y)\big|_{\mathbf{x} = x+z}
\sqrt{\Xi_{\a}(x,z,t)}- \partial_t^m \delta_{\eta, n, \om, \mathbf{x}} \mathfrak{G}_{t}^{\a,\eta,+}(\mathbf{x},y)
\big|_{\mathbf{x} = x'+z}\sqrt{\Xi_{\a}(x',z,t)}\right| \\
 \qquad {}\le \left| \partial_t^m \delta_{\eta, n, \om, \mathbf{x}} \mathfrak{G}_{t}^{\a,\eta,+}(\mathbf{x},y)\big|_{\mathbf{x} = x+z}-
\partial_t^m \delta_{\eta, n, \om, \mathbf{x}} \mathfrak{G}_{t}^{\a,\eta,+}(\mathbf{x},y)
\big|_{\mathbf{x} = x'+z}\right| \sqrt{\Xi_{\a}(x',z,t)} \\
 \qquad\quad{} + \left| \partial_t^m \delta_{\eta, n, \om, \mathbf{x}} \mathfrak{G}_{t}^{\a,\eta,+}(\mathbf{x},y)
\big|_{\mathbf{x} = x+z}\right| \left| \sqrt{\Xi_{\a}(x,z,t)} - \sqrt{\Xi_{\a}(x',z,t)} \right| \\
 \qquad {}\equiv I_1(x,x',y,z,t) + I_2(x,x',y,z,t).
\end{gather*}
We treat $I_1$ and $I_2$ separately. Using successively the mean value theorem, Lemma~\ref{lem:heatEST} (with $r = \boldsymbol{0}$ and $l = e_i$, $i = 1, \ldots, d$), \eqref{ineq2}, \eqref{ineq1}, \eqref{ineq3} and f\/inally Lemma~\ref{lem:theta} twice (f\/irst with $z = \t$ and then with $z = x \vee x'$)
we arrive at
\begin{gather*}
 I_1(x,x',y,z,t) \\ \qquad{} \lesssim
\|x - x'\| \sum_{\substack{ \eps, \rho, \xi \in \{0,1\}^d \\ a, b \in \{0,1,2\}^d }}
\sum_{\boldsymbol{0} \le \tau \le \eta - \rho\eta + 2\eps - a\eps }
(x \vee x')^{\eta - \rho\eta + 2\eps - a\eps - \tau } y^{\eta - \xi \eta + 2\eps - b \eps}
\big( \lo (\z) \big)^{|\tau|/2} \\
 \qquad \quad{} \times \big(1-\zeta^2\big)^{d + |\a| + |\eta| + 2|\eps| - D} \zeta^{-d - |\a| - |\eta| - 2|\eps| - m - |n|/2 -1/2 +
(|\rho \eta| + |a \eps| + |\xi \eta| + |b \eps| )/2 } \\
\qquad \quad {} \times 	\int \big( \e\eee \big)^{D/16} \piint \sqrt{\Xi_{\a}(x',z,t)},
\end{gather*}
provided that $(z,t) \in A_1$ and $\|x-y\|>2\|x-x'\|$. Now the conclusion for $I_1$ follows from Lemma~\ref{lem:intLusin}(a), Lemma~\ref{lem:bridge}
(specif\/ied to $p = 2$, $W = |n| + 2m$, $C = D/16$ and $u = 1$) and Lemma~\ref{lem:double} (applied with $\gamma=1$ and $z=x \vee x'$).

To estimate the norm of $I_2$ we use \eqref{ineq4} and Lemma~\ref{lem:intLusin}(b) to obtain
\begin{gather}\nonumber
\left\| I_2(x,x',y,z,t) \right\|_{L^2(A_1,t^{|n| + 2m -1} dzdt)} \\
 \qquad {} \lesssim \|x - x'\|^{\gamma} \!\sum_{\substack{ \eps, \rho, \xi \in \{0,1\}^d \\
a, b \in \{0,1,2\}^d }}\sum_{\boldsymbol{0} \le \tau \le \eta - \rho\eta + 2\eps - a\eps }
\!\Big\|
x^{\eta - \rho\eta + 2\eps - a\eps - \tau } y^{\eta - \xi \eta + 2\eps - b \eps}
\big( \lo (\z) \big)^{|\tau|/2}\!\!\label{ineq5} \\
 \qquad \quad {} \times \big(1-\zeta^2\big)^{d + |\a| + |\eta| + 2|\eps| - D}\zeta^{-d - |\a| - |\eta| - 2|\eps| - m - |n|/2 +
(|\rho \eta| + |a \eps| + |\xi \eta| + |b \eps| )/2 -\gamma/2}\nonumber \\ \nonumber
 \qquad \quad{} \times	\int \big( \e\bbb \big)^{D} \piint
\Big\|_{L^2(\R_+,t^{|n| + 2m -1} dt)}, \qquad x,x',y \in \RRR.
\end{gather}
This, however, in view of Lemma~\ref{lem:bridge} (taken with $p = 2$, $W = |n| + 2m$, $C = D$ and $u = \gamma$) gives the desired estimate for $I_2$ and therefore f\/inishes the analysis related to $A_1$.

{\bf Case 2:} \textbf{The norm related to $L^2(A_{2},t^{|n| + 2m - 1}dzdt).$} Since $\mathfrak{S}^{\alpha, \eta, +}_{n,m,\om}(x',y) =0$, our aim is to prove that
\begin{gather}\label{ineq6}
\|\mathfrak{S}^{\alpha, \eta, +}_{n,m,\om}(x,y)\|_{L^2(A_{2},t^{|n| + 2m - 1}dzdt)}
\lesssim \left(\frac{\|x-x'\|}{\|x-y\|} \right)^{\gamma} \frac{1}{\mu_\a^+(B(x,\|x-y\|))},
\end{gather}
for $\|x-y\|>2\|x-x'\|$. Taking into account \eqref{ineq4} and then applying Lemma~\ref{lem:intLusin}(c) we get{\samepage
\begin{gather*}
\| \mathfrak{S}^{\alpha, \eta, +}_{n,m,\om}(x,y)\|_{L^2(A_{2},t^{|n| + 2m - 1}dzdt)} \\
 \qquad{} \lesssim \|x-x'\|^{\gamma} \sum_{\substack{ \eps, \rho, \xi \in \{0,1\}^d \\
a, b \in \{0,1,2\}^d }}\sum_{\boldsymbol{0} \le \tau \le \eta - \rho\eta + 2\eps - a\eps }
\Big\| x^{\eta - \rho\eta + 2\eps - a\eps - \tau } y^{\eta - \xi \eta + 2\eps - b \eps}\big( \lo (\z) \big)^{|\tau|/2} \\
 \qquad \quad {}\times \big(1-\zeta^2\big)^{d + |\a| + |\eta| + 2|\eps| - D} \zeta^{-d - |\a| - |\eta| - 2|\eps| - m - |n|/2 +
(|\rho \eta| + |a \eps| + |\xi \eta| + |b \eps| )/2 -\gamma/2} \\
 \qquad \quad \times	\int \big( \e\bbb \big)^{D} \piint \Big\|_{L^2(\R_+,t^{|n| + 2m - 1}dt)}, \qquad x,x',y \in \RRR.
\end{gather*}
The right-hand side here coincides with the right-hand side of \eqref{ineq5}, and \eqref{ineq6} follows.}

Finally, we focus on the second smoothness condition \eqref{sm2}. We will prove it with $\gamma = 1$. Applying sequently the mean value theorem,
Lemma~\ref{lem:heatEST} (choosing $r = e_i$, $l = \boldsymbol{0}$, $i = 1, \ldots, d$), \eqref{ineq2}, \eqref{ineq3} and then Lemma~\ref{lem:theta} twice (f\/irst with $z = \t$ and then with $z = y \vee y'$) we obtain
\begin{gather*}
\left| \partial_t^m \delta_{\eta, n, \om, \mathbf{x}} \mathfrak{G}_{t}^{\a,\eta,+}(\mathbf{x},y)
\big|_{\mathbf{x} = x+z} - \partial_t^m \delta_{\eta, n, \om, \mathbf{x}} \mathfrak{G}_{t}^{\a,\eta,+}(\mathbf{x},y')
\big|_{\mathbf{x} = x+z} \right| \sqrt{\Xi_{\a}(x,z,t)} \chi_{\{x+z\in\RRR\}} \\
 \qquad {} \lesssim \| y - y' \| \sum_{\substack{ \eps, \rho, \xi \in \{0,1\}^d \\
a, b \in \{0,1,2\}^d }} \sum_{\boldsymbol{0} \le \tau \le \eta - \rho\eta + 2\eps - a\eps }
x^{\eta - \rho\eta + 2\eps - a\eps - \tau } (y \vee y')^{\eta - \xi \eta + 2\eps - b \eps}
\big( \lo (\z) \big)^{|\tau|/2} \\
 \qquad \quad {} \times \big(1-\zeta^2\big)^{d + |\a| + |\eta| + 2|\eps| - D}
\zeta^{-d - |\a| - |\eta| - 2|\eps| - m - |n|/2 - 1/2 +
(|\rho \eta| + |a \eps| + |\xi \eta| + |b \eps| )/2 } \\
 \qquad \quad{} \times	\int \big( \e\fff \big)^{D/16} \piint 	\sqrt{\Xi_{\a}(x,z,t)}\chi_{\{x+z\in\RRR\}},
\end{gather*}
provided that $(z,t) \in A$ and $\|x-y\|>2\|y-y'\|$. The required estimate follows by using Lemma~\ref{lem:intLusin}(a), Lemma~\ref{lem:bridge}
(specif\/ied to $p = 2$, $W = |n| + 2m$, $C = D/16$ and $u = 1$) and f\/inally Lemma~\ref{lem:double}.

The proof of Theorem~\ref{thm:kerest} is complete.
\end{proof}

\begin{proof}[Proof of Theorem~\ref{Sthm:kerest}]
Since the estimates of various derivatives of the heat kernels in the Laguerre--Dunkl and the Laguerre-symmetrized settings established in Lemma~\ref{lem:heatEST} are the same, the proof of Theorem~\ref{Sthm:kerest} is just a repetition of the arguments given in the proof of Theorem~\ref{thm:kerest} above. Therefore we omit the details.
\end{proof}

\appendix

\section{Appendix I} \label{sec:App}

In this short section we shall prove the following useful result.
\begin{Proposition}\label{pro:conv}
Let $1\le p<\infty$. If $W\in A^{\a}_p$ and $f\in L^p(W d\mu_{\a})$, then the series/integrals defining $\mathfrak{T}_t^{\alpha}f(x)$ and $\mathbb{T}_t^{\alpha} f(x)$ converge for every $x\in\RR$ and $t>0$ and produce smooth functions of $(x,t)\in \RR \times\mathbb{R}_{+}$. Similarly, if $U\in A^{\a,+}_p$ and $f\in L^p(U d\mu_{\a}^+)$, then the series/integrals defining $\mathfrak{T}_t^{\a,\eta,+}f(x)$ and $\mathbb{T}_t^{\alpha,\eta, +} f(x)$,
$\eta\in\{0,1\}^d$, converge for every $x\in\RRR$ and $t>0$ and produce smooth functions of $(x,t)\in \RRR \times\mathbb{R}_{+}$.
\end{Proposition}

\begin{proof} Recall that the def\/inition of $\mathfrak{T}_t^{\alpha}g$ for $g\in L^2( d\mu_{\a})$ is
\begin{gather} \label{srepr}
\mathfrak{T}_t^{\alpha} g(x) = \sum_{k \in \N^d} e^{-t \lambda_{|k|/2}^{\a}}
\langle g , h_k^{\a} \rangle_{d\mu_{\alpha}}h_k^{\a}(x),
\end{gather}
(convergence in $L^2( d\mu_{\a})$) and this easily leads to the integral representation
\begin{gather} \label{repr}
\mathfrak{T}_t^{\a} g (x)=\int_{\RRR} \mathfrak{G}_t^{\a} (x,y) g(y) \, d\mu_\a(y), \qquad g\in L^2( d\mu_{\a}),\qquad x\in \RR, \qquad t>0.
\end{gather}
To prove the claim for $\mathfrak{T}_t^{\alpha}f(x)$ we use the following two auxiliary results. Firstly, given $1\le p<\infty$, $W\in A^{\a}_p$ and $f\in L^p(W d\mu_{\a})$, the coef\/f\/icients $\langle f , h_k^{\a} \rangle_{d\mu_{\alpha}}$, $k \in \N^d$, exist and satisfy
\begin{gather} \label{esti1}
|\langle f , h_k^{\a} \rangle_{d\mu_{\alpha}}| \lesssim (|k|+1)^{c_{d,\a,p,W}}\|f\|_{L^p(W d\mu_{\a})},
\end{gather}
uniformly in $k \in \N^d$ and $f\in L^p(W d\mu_{\a})$. Secondly,
\begin{gather} \label{esti2}
|h_k^{\a}(x)| \lesssim (|k|+1)^{c_{d,\a,p}},
\end{gather}
uniformly in $k \in \N^d$ and $x\in \RR$. Then, following the argument from \cite[pp.~647--648]{NoSt2} one checks that after replacing $g\in L^2( d\mu_{\a})$ by $f\in L^p(W d\mu_{\a})$ in the right-hand side of \eqref{srepr} the series converges absolutely for any $x\in \RR$, $t>0$, and thus def\/ines $\mathfrak{T}_t^{\alpha} f(x)$. Moreover, with this def\/inition of $\mathfrak{T}_t^{\a} f (x)$ the integral representation~\eqref{repr} remains valid for $f$ replacing $g$ (in particular, the relevant integral converges for any $x\in \RR$ and $t>0$) and $\mathfrak{T}_t^{\alpha}f(x)$ is a $C^\infty$ function of $(x,t)\in \RR \times\mathbb{R}_{+}$.

Coming back to \eqref{esti1} and \eqref{esti2}, these are simple consequences of \cite[equations~(2.3) and~(2.4)]{TZS3}. The assumption $\a\in[-1/2,\infty)^d$ which was imposed in \cite{TZS3} is not essential for \cite[equations~(2.3) and (2.4)]{TZS3} to hold since the classical estimates for the standard Laguerre functions due to Askey, Wainger and Muckenhoupt, invoked in \cite[p.~1521]{TZS3}, are valid for any Laguerre type parameter greater than $-1$, thus $\a\in(-1,\infty)^d$ is admitted.

The claim for $\mathfrak{T}_t^{\a,\eta,+}f(x)$ follows due to the connection $\mathfrak{T}_t^{\a,\eta,+}f(x) = \big( \mathfrak{T}_t^{\a} f^\eta \big)^+(x)$, $x\in\RRR$, which holds for any $f\in L^p(U d\mu_{\a}^+)$, $U\in A^{\a,+}_p$, $1\le p<\infty$ (note that then $f^\eta \in L^p(W d\mu_{\a})$, where $W = U^{\boldsymbol{0}} \in A^{\a}_p$). Finally, the claims for $\mathbb{T}_t^{\alpha} f(x)$ and $\mathbb{T}_t^{\alpha,\eta, +} f(x)$ are verif\/ied by arguments analogous to those just presented.
\end{proof}

\section{Appendix II} \label{sec:App2}

For reader's convenience, in Table~\ref{table1} below we summarize the notation of various objects in the three contexts appearing in this paper.

\begin{table}[!ht] \centering\vspace*{-2mm}
	\caption{Summary of notation.}\label{table1} \vspace{1mm}
 \begin{tabular}{|c|c|c|c|} \hline
		 & Laguerre--Dunkl & Laguerre-symmetrized & Laguerre \tsep{1pt}\bsep{1pt}\\
\hline
Harmonic oscillator & $\mathfrak{L}_{\a}$ & $\mathbb{L}_{\a}$ & $L_{\a}$ \tsep{1pt}\bsep{1pt}\\
\hline
Eigenfunctions & $h_k^\a$ & $\Phi_k^\a$ & $\ell_k^\a$ \tsep{1pt}\bsep{1pt}\\
\hline
Reference measure & $\mu_{\a}$ & $\mu_{\a}$ & $\mu_{\a}^+$ \tsep{1pt}\bsep{1pt}\\
\hline
Derivatives & $\mathfrak{D}_i$, $\mathfrak{D}_i^{*}$, $\mathfrak{D}^{n,\omega}$
& $\mathbb{D}_i$, $\mathbb{D}^n$ & $\delta_i$, $\delta_i^*$, $\delta^n$, $D^n$ \tsep{1pt}\bsep{1pt}\\
\hline
Heat semigroup & $\mathfrak{T}_t^\a$ & $\mathbb{T}_t^\a$ & $T_t^\a$ \tsep{1pt}\bsep{1pt}\\
\hline
Heat kernel & $\mathfrak{G}_t^\a (x,y)$ & $\mathbb{G}_t^\a (x,y)$ & $G_t^\a (x,y)$ \tsep{1pt}\bsep{1pt}\\
\hline
Maximal operator & $\mathfrak{T}_*^\a$ & $\mathbb{T}_*^\a$ & \tsep{1pt}\bsep{1pt}\\
\hline
Riesz transforms & $\mathfrak{R}_{n,\omega}^\a$ & $\mathbb{R}_{n}^\a$ & $R_n^\a$ \tsep{1pt}\bsep{1pt}\\
\hline
Multipliers & $\mathfrak{M}_{\mathfrak{m}}^\a$ & $\mathbb{M}_{\mathfrak{m}}^\a$ & \tsep{1pt}\bsep{1pt}\\
\hline
$g$-functions & $\mathfrak{g}_{n,m,\omega}^\a$ & $\mathbb{g}_{n,m}^\a$ & $g_{n,m}^\a$ \tsep{1pt}\bsep{1pt}\\
\hline
Lusin area integrals & $\mathfrak{S}_{n,m,\omega}^\a$ & $\mathbb{S}_{n,m}^\a$ & $S_{n,m}^\a$, $s_{n,m}^\a$\tsep{1pt}\bsep{1pt}\\
\hline
Main results
& Theorems \ref{thm:main}, \ref{thm:mainP}
& Theorem~\ref{Sthm:main} & Theorems \ref{thm:lag}, \ref{thm:lags}\tsep{1pt}\bsep{1pt}\\
\hline
\end{tabular}\vspace{-3mm}
\end{table}

\subsection*{Acknowledgements}
Research of the f\/irst-named and the second-named authors was supported by the National Scien\-ce Centre of Poland, project no.~2013/09/B/ST1/02057. The third-named author was partially supported by the National Science Centre of Poland, project no.~2012/05/N/ST1/02746.

\pdfbookmark[1]{References}{ref}
\LastPageEnding

\end{document}